\documentclass[10pt, reqno]{amsart}
\usepackage[colorlinks=true, pdfstartview=FitV, linkcolor=blue, citecolor=blue, urlcolor=blue]{hyperref}
\usepackage{amssymb, amsmath, amsthm}
\usepackage{enumitem}
\usepackage[all]{xy}
\usepackage[normalem]{ulem}
\usepackage{stmaryrd}

\newtheorem{theorem}{Theorem}[section]
\newtheorem*{theorem*}{Theorem}
\newtheorem{maintheorem}{Main Theorem}
\newtheorem{lemma}[theorem]{Lemma}
\newtheorem{corollary}[theorem]{Corollary}
\newtheorem{proposition}[theorem]{Proposition}

\theoremstyle{remark}
\newtheorem{remark}[theorem]{Remark}
\newtheorem{example}[theorem]{Example}

\theoremstyle{definition}
\newtheorem{definition}[theorem]{Definition}

\DeclareMathOperator{\SL}{SL}

\DeclareMathOperator{\Lie}{Lie}

\DeclareMathOperator{\Aut}{Aut}

\DeclareMathOperator{\id}{id}

\DeclareMathOperator{\Spec}{Spec}

\DeclareMathOperator{\Hom}{Hom}

\DeclareMathOperator{\rank}{rank}

\DeclareMathOperator{\aff}{aff}

\DeclareMathOperator{\Ve}{Vec}
\DeclareMathOperator{\LND}{LND}

\DeclareMathOperator{\Conv}{Conv}
\DeclareMathOperator{\Span}{Span}
\DeclareMathOperator{\Mor}{Mor}
\DeclareMathOperator{\Der}{Der}
\DeclareMathOperator{\inter}{int}

\newcommand{\norm}[1]{\left\lVert#1\right\rVert}
\newcommand{\OO}{\mathcal{O}}
\newcommand{\GG}{\mathbb{G}}
\newcommand{\RR}{\mathbb{R}}
\newcommand{\PP}{\mathbb{P}}
\newcommand{\ZZ}{\mathbb{Z}}
\renewcommand{\AA}{\mathbb{A}}

\newcommand{\kk}{\textbf{k}}

\newcommand{\aquot}{/ \! \! /}

\newcommand{\name}[1]{#1}

\renewcommand{\phi}{\varphi}

\newcommand{\set}[2]{\left\{\,#1 \ | \ #2\,\right\}}
\newcommand{\Bigset}[2]{\left\{\,#1 \ \Big| \ #2\,\right\}}
\newcommand{\sprod}[2]{\langle #1, #2 \rangle}

\makeatletter
\ifcsname phantomsection\endcsname
\newcommand*{\qrr@gobblenexttocentry}[5]{}
\else
\newcommand*{\qrr@gobblenexttocentry}[4]{}
\fi
\newcommand*{\addsubsection}{%
	\addtocontents{toc}{\protect\qrr@gobblenexttocentry}%
	\subsection}
\makeatother

\title[Quasi-affine spherical varieties via the automorphism group]
{Characterizing quasi-affine spherical varieties via the automorphism group}
\author[A. Regeta \and I. van Santen]
{Andriy Regeta \and Immanuel van Santen}

\thanks{}

\address{\noindent Institut f\"{u}r Mathematik, Friedrich-Schiller-Universit\"{a}t Jena, \newline
\indent  Jena 07737, Germany}
\email{andriyregeta@gmail.com}

\address{Departement Mathematik und Informatik, 
Universit\"at Basel,\newline
\indent Spiegelgasse 1, CH-4051 Basel, Switzerland}
\email{immanuel.van.santen@math.ch}

\begin{document}


\maketitle

\begin{abstract}
	Let $G$ be a connected reductive algebraic group.
	In this  note we prove that for a quasi-affine $G$-spherical variety the weight monoid is
	determined by the weights of its non-trivial $\mathbb{G}_a$-actions that are homogeneous with
	respect to a Borel subgroup of $G$. 
	As an application we get that  a smooth affine $G$-spherical variety that is non-isomorphic to a torus 
	is determined by its automorphism group inside the category of smooth affine irreducible varieties.
\end{abstract}

\tableofcontents

\section{Introduction}

In \cite[Theorem~1.1]{Kraft2017}, \name{Kraft} proved 
that $\mathbb{A}^n$ is determined by 
its automorphism group $\Aut(\AA^n)$ seen as an ind-group inside 
the category of connected affine varieties
and in \cite[Main Theorem]{KrReSa2018Is-the-affine-spac}, 
this result was partially generalized in case $\Aut(\AA^n)$ is seen only as an abstract group. 
In \cite[Theomrem~A]{CaReXi2019Families-of-Commut}, 
the last results are widely generalized in the following sense:
$\AA^n$ is completely characterized through the abstract group
$\Aut(\AA^n)$ inside the category of connected affine varieties.
The result of \name{Kraft} was partially generalized to other affine varieties
than the affine space in \cite{Re2017Characterization-o} and \cite{LiReUr2018Characterization-o}. 
More precisely,  there is the following statement: 

\begin{theorem}[{\cite[Theorem 1.4]{LiReUr2018Characterization-o}}]

\label{theorem0}
Let $X$ be an affine toric variety different from the torus
and let $Y$ be an irreducible normal affine variety. If $\Aut(X)$ and $\Aut(Y)$ are isomorphic
as an ind-groups, then $X$ and $Y$ are isomorphic as varieties.  
\end{theorem}

\begin{remark}
 In fact, in both \cite{Kraft2017} and \cite[Theorem 1.4]{LiReUr2018Characterization-o}, the
 authors prove the statements under the slightly weaker assumption 
 that there is a group isomorphism $\Aut(X) \simeq \Aut(Y)$ 
 that preserves algebraic subgroups (see Sect.~\ref{sec.AutomorphismGroups} for the definition).
\end{remark}

In this article, we work over an algebraically closed field $\kk$ of characteristic zero.
A natural generalization of toric varieties are the so-called spherical varieties.
Let $G$ be a connected reductive algebraic group.
Recall that a variety $X$ endowed with a faithful $G$-action
is called \emph{$G$-spherical} if  some (and hence every) 
Borel subgroup in $G$ acts on $X$ with an open dense orbit, see e.g. \cite{Ti2011Homogeneous-spaces} for a survey on the topic. If $G$ is a torus, then a $G$-spherical variety is the same thing as a $G$-toric variety.
If $X$ is $G$-spherical, then 
$X$ has an open $G$-orbit which is isomorphic to $G/H$ for some subgroup $H \subset G$.
The family of $G$-spherical varieties is, in a sense,  the widest family of $G$-varieties which is well-studied: in fact, $G$-equivariant open embeddings of $G$-homogeneous 
$G$-spherical varieties are classified
by certain combinatorial data (analogous to the classical case of toric varieties)
by \name{Luna-Vust} \cite{LuVu1983Plongements-despac} (see also the work of \name{Knop} \cite{Kn1991The-Luna-Vust-theo}) 
and homogeneous $G$-spherical varieties are classified for $\kk$ equal to the complex numbers 
by \name{Luna},  \name{Bravi}, \name{Cupit-Foutou}, \name{Losev} and \name{Pezzini} 
\cite{Lu2001Varietes-spherique, BrPe2005Wonderful-varietie, Br2007Wonderful-varietie,
Lu2007La-variete-magnifi, Lo2009Uniqueness-propert, BrCu2010Classification-of-, Cu2014Wonderful-Varietie}.
%

 In this paper, we  generalize partially Theorem  \ref{theorem0} to quasi-affine $G$-spherical varieties.  
 In order to state our main results, let us introduce some notation. Let
 $X$ be an irreducible $G$-variety for a connected algebraic group $G$ with a fixed Borel subgroup
 $B \subset G$.
 We denote by $\frak{X}(B)$ the character group
 of $B$, i.e. the group of regular group homomorphisms $B \to \GG_m$.  
 The \emph{weight monoid} of $X$ is defined by
 \[
	\Lambda^+(X) = \{ \,\lambda \in \frak{X}(B) \ | \ \OO(X)_{\lambda}^{(B)} \neq 0 \, \}
 \] 
where $\OO(X)_{\lambda}^{(B)} \subset \OO(X)$ 
denotes the subspace of $B$-semi-invariants of weight $\lambda$
of the coordinate ring $\OO(X)$.
Our main result of this article is the following:




\begin{maintheorem}
	\label{mainthm:autos}
	Let $X$, $Y$ be irreducible normal quasi-affine varieties, let
	$\theta \colon \Aut(X) \simeq \Aut(Y)$ be a group isomorphism
	that preserves algebraic subgroups (see Sect.~\ref{sec.AutomorphismGroups} for the definition) 
	and let $G$ be a connected reductive 
	algebraic group. 
	Moreover, we fix a Borel subgroup $B \subset G$. 
	If $X$ is $G$-spherical and not isomorphic to a torus, then the following holds:
	\begin{enumerate}[leftmargin=*]
		\item \label{main1} $Y$ is $G$-spherical for the induced $G$-action via $\theta$;
		\item \label{main2} the weight monoids $\Lambda^+(X)$ and $\Lambda^+(Y)$ inside $\frak{X}(B)$ 
		are the same;
		\item \label{main3} if one of the following assumptions holds
									 \begin{itemize}[leftmargin=*]
									 	\item $X$, $Y$ are smooth and affine or
									 	\item $X$, $Y$ are affine and $G$ is a torus,
									 \end{itemize}
									 then $X$ and $Y$ are isomorphic as $G$-varieties.
	\end{enumerate}
\end{maintheorem}

We prove Main Theorem~\ref{mainthm:autos}\eqref{main1} in Proposition~\ref{prop.iso_and_sphericity},
Main Theorem~\ref{mainthm:autos}\eqref{main2} in Corollary~\ref{thesameweights} 
and Main Theorem~\ref{mainthm:autos}\eqref{main3} in Theorem~\ref{thm.smooth_spherical_var}.

In case $X$ is isomorphic to a torus and $X$ is $G$-spherical, it follows that $G$ is in fact a torus
of dimension $\dim X$, as by assumption $G$ acts faithfully on $X$. Thus $X\simeq G$. 
Then \cite[Example 6.17]{LiReUr2018Characterization-o} gives an example
of an affine variety $Y$ such that there is a group isomorphism $\theta \colon \Aut(X) \to \Aut(Y)$
that preserves algebraic subgroups, but $Y$ is not $G$-toric. Thus the assumption that $X$ is
not isomorphic to a torus in Main Theorem~\ref{mainthm:autos} is essential.


%
  
Moreover, in general, we cannot drop the normality condition in Main Theorem~\ref{mainthm:autos}:
We give an example in Lemma~\ref{lem.ex_normality} where the weight monoids of $X$ and $Y$ are different, 
see Sect.~\ref{sec.Counter_example}.

To prove Main Theorem \ref{mainthm:autos}, 
we study so-called generalized root subgroups of $\Aut(X)$ and 
their weights  for a $G$-variety $X$ (see Sect.~\ref{sec.determination_of_sphericity} for details). 
We show that if a Borel subgroup $B$ of $G$ is not a torus, 
then an irreducible normal quasi-affine variety with a faithful $G$-action is
$G$-spherical if and only if the dimension of all 
generalized root subgroups of $\Aut(X)$ with respect to $B$ 
is bounded (see Proposition~\ref{prop:characterization_of_aff_spherical_var}). 

Moreover, we prove that the weight monoid $\Lambda^+(X)$ of a quasi-affine 
$G$-spherical variety $X$ is encoded in 
the set of $B$-weights of non-trivial $B$-homogeneous $\GG_a$-actions:
\[
	D(X) = \Bigset{ \lambda \in \frak{X}(B)}{ 
		\begin{array}{l}
		\textrm{there exists a non-trivial $B$-homogeneous} \\
		\textrm{$\GG_a$-action on $X$ of weight $\lambda$} 
		\end{array}
	}
\]
(see Subsec.~\ref{sec.GroupActionsAndVectorfields} for the definition of a $B$-homogeneous $\GG_a$-action).
To $D(X) \subset \frak{X}(B)$ we may associate its asymptotic cone $D(X)_{\infty}$ inside 
$\frak{X}(B) \otimes_\ZZ \RR$
(see Sect.~\ref{sec.Cones_and_asymptotic_cones} for the definition).
We prove then the following ``closed formula" for the weight monoid:

\begin{maintheorem}
	\label{mainthm:closed_formula}
	Let $G$ be a connected reductive algebraic group, let $B \subset G$ be a Borel subgroup and let
	$X$ be a quasi-affine $G$-spherical variety which is non-isomorphic to a torus.
	If $G$ is not a torus or $\Spec(\OO(X)) \not\simeq \AA^1 \times (\AA^1 \setminus \{0\})^{\dim(X)-1}$, then
	\[
		\Lambda^+(X) = \Conv ( D(X)_{\infty}) \cap \Span_{\ZZ}(D(X))
	\]
	where the asymptotic cones and linear spans are taken inside $\frak{X}(B) \otimes_{\ZZ} \RR$.
\end{maintheorem}

Main Theorem~\ref{mainthm:closed_formula} is proven in Theorem~\ref{thm.description_of_weight-monoid}. Using
this result, we then obtain as a consequence:

\begin{maintheorem}
	Let $G$ be a connected reductive algebraic group and let
	$X, Y$ be quasi-affine $G$-spherical varieties with $D(X) = D(Y)$. Then $\Lambda^+(X) = \Lambda^+(Y)$.
\end{maintheorem}

This is proven in Corollary~\ref{cor.weights_describe_weight-monoid}.

\addsubsection*{Acknowledgements}
The authors would like to thank \name{Michel Brion} for giving us the idea to
study asymptotic cones, which eventually led to a proof of our main result.

\section{Cones and asymptotic cones}
\label{sec.Cones_and_asymptotic_cones}
In the following section we introduce some basic facts about cones and convex sets.
As a reference for cones we take~\cite[\S1.2]{Fu1993Introduction-to-to} and as a 
reference for asymptotic cones we take~\cite[Chp.~2]{AuTe2003Asymptotic-cones-a}.

Throughout this section $V$ denotes a Euclidean vector space, i.e. a finite dimensional $\RR$-vector space 
together with a scalar product
\[
	V \times V \to \RR \, , \quad (u, v) \mapsto \langle u, v \rangle \, .
\]
The induced norm on $V$ we denote by $\norm{\cdot} \colon V \to \RR$.

A subset $C \subset V$ is a \emph{cone} if 
for all $\lambda \in \RR_{\geq 0}$ and for all $c \in C$ we have 
$\lambda \cdot c \in C$. 
The \emph{asymptotic cone} $D_{\infty}$
of a subset $D \subset V$ is defined as follows
\[
D_{\infty} = \left\{ x \in V \ \Big| \  
\begin{array}{l}
\textrm{there exists a sequence $(x_i)_i$ in $D$ with $\norm{x_i} \to \infty$} \\
\textrm{such that $\lim_{i \to \infty} \RR x_i = \RR x$ inside $\PP(V)$} 
\end{array}
\right\}
\]
where $\PP(V)$ denotes the space of lines through the origin in $V$.
The asymptotic cone satisfies the following basic properties, see e.g. 
\cite[Proposition~2.1.1, Proposition~2.1.9]{AuTe2003Asymptotic-cones-a}.
\begin{lemma}[Properties of asymptotic cones] 
	\label{lem.Properties_asymptotic_cone}
	$ $
	\begin{enumerate}[leftmargin=*]
		\item If $D \subset V$, then $D_{\infty} \subset V$
		is a closed cone.
		\item If $C \subset V$ is a closed cone, then $C_{\infty} = C$. 
		\item If $D \subset D' \subset V$, then $D_\infty \subset (D')_{\infty}$.
		\item If $D \subset V$ and $v \in V$, then 
		$(v+D)_{\infty} = D_{\infty}$.
		\item If $D_1, \ldots, D_k \subset V$, then 
		$(D_1 \cup \ldots \cup D_k)_{\infty} =
		(D_1)_{\infty} \cup \ldots \cup (D_k)_{\infty}$.
	\end{enumerate}
\end{lemma}

Moreover, we need the following property:
\begin{lemma}[Asymptotic cone of a $\delta$-neighbourhood]
	\label{lem.Asymptotic_cone_delt_neighb}
	Let $D \subset V$ and let $\delta \in \RR$ with $\delta \geq 0$. Then 
	the $\delta$-neighbourhood of $D$
	\[
		D^{\delta} := \set{x \in V}{\textrm{there is $y \in D$ with $\norm{x-y} \leq \delta$ }}
	\]
	satisfies $(D^{\delta})_{\infty} = D_{\infty}$.
\end{lemma}

\begin{proof}
	We only have to show that $(D^{\delta})_{\infty} \subset D_{\infty}$. Let $x \in (D^{\delta})_{\infty}$ and let 
	$(x_i)_i$ be a sequence
	in $D^{\delta}$ such that $\norm{x_i} \to \infty$ and $\lim_{i \to \infty} \RR x_i = \RR x$. By definition, 
	there is a sequence $(y_i)_i$ in $D$ such that $\norm{x_i-y_i} \leq \delta$. In particular, we get
	$\norm{y_i} \to \infty$. Let $m_i := \min\{\norm{x_i}, \norm{y_i}\}$. Then $m_i \to \infty$ for $i \to \infty$
	and for big enough $i$ we get
	\[
		0 \leq \norm{\frac{x_i}{\norm{x_i}}-\frac{y_i}{\norm{y_i}}} \leq \frac{\norm{x_i-y_i}}{m_i} \leq \frac{\delta}{m_i} \, .
	\]
	As $\frac{\delta}{m_i} \to \infty$ for $i \to \infty$, the above inequality
	implies $\RR x = \lim_{i \to \infty} \RR x_i = \lim_{i \to \infty} \RR y_i$ inside $\PP(V)$.
\end{proof}

For the next statement, we recall some definitions.
A subset $C \subset V$ is called \emph{convex} if for all
$x, y \in C$ and all $\alpha \in [0, 1]$, we have $\alpha x + (1- \alpha) y \in C$.
A convex cone $C \subset V$  is called \emph{strongly convex}, if
it contains no linear subspace of $V$ except the zero subspace.
For a
subset $D \subset V$, we denote by $\Conv(D)$ the 
\emph{convex cone generated by $D$} in $V$, i.e.
\[
\Conv(D) = \set{\lambda_1 v_1 + \ldots + \lambda_k v_k \in V}{
	\textrm{$v_1, \ldots, v_k \in D$ and $\lambda_1, \ldots, \lambda_k \in \RR_{\geq 0}$}} \, .
\]

A subset $C \subset V$ is a \emph{convex polyhedral cone} if there is a finite subset 
$F \subset V$ such that
\[
C = \Conv F \, .
\]

A hyperplane $H \subset V$ passing through the origin is called a \emph{supporting hyperplane} of a convex polyhedral cone 
$C \subset V$ if
$C$ is contained in one of the closed half spaces in $V$ 
induced by $H$, i.e. there is a normal vector $u \in V$ to $H$
such that
\[
	C \subset \set{x \in V}{\langle u, x \rangle\geq 0} \, .
\]
A \emph{face} of a convex polyhedral cone $C \subset V$ is the intersection of $C$ with a
supporting hyperplane of $C$ in $V$.
A face of dimension one of $C$ is called an \emph{extremal ray of $C$}. 
For a subset $D \subset V$ we denote by $\inter(D)$ the topological interior 
of $D$ inside the linear span of $D$.

\begin{lemma}
	\label{lem.Intersection_with_affine_hyperplane}
	Let $C \subset V$ be a cone and let $H_1$ be an affine hyperplane in $V$
	such that $0 \not\in H_1$. If $C \cap H_1 \neq \varnothing$, then $\inter(C) \cap H_1 \neq \varnothing$.
\end{lemma}

\begin{proof}
	Let $\pi \colon V \to \RR$ be a linear map such that $H_1 = \pi^{-1}(1)$.
	By assumption, there is $c \in C \cap H_1$.
	We may assume that $c$ lies in the topological boundary of $C$ inside the linear span of $C$
	(otherwise we are finished). 
	By the continuity 
	of $\pi$, there is $c' \in \inter(C)$ such that $|\pi(c)- \pi(c')| < 1$. As $\pi(c) = 1$, we
	get $\pi(c') > 0$. Then $\lambda = 1 / \pi(c') \in \RR_{> 0}$ and thus 
	$\lambda c' \in \inter(C) \cap H_1$. 
\end{proof}





\begin{lemma}
	\label{lem.Weights_facet_translated}
	Let $C \subset V$ be a convex polyhedral cone and let 
	$H \subset V$ be a hyperplane. Then for each $v \in V$ such that $C \cap (v+H) \neq \varnothing$, we have
	\[
		\left( C \cap (v+H) \right)_{\infty} = C \cap H \, .
	\]
\end{lemma}

\begin{proof}	
	We denote $D := C \cap (v+H) \subset V$. As 
	$D \neq \varnothing$ we can take $x \in D$. If $y \in C \cap H$, then
	$x+y \in C$ and $x+y \in v + H$, thus $x+y \in D$. This shows that
	$x + (C \cap H) \subset D$ and by Lemma~\ref{lem.Properties_asymptotic_cone}, we get $C \cap H \subset D_\infty$.
	Now, from Lemma~\ref{lem.Properties_asymptotic_cone} we get also the reverse inclusion:
	\[
		D_\infty = (C \cap (v+H))_\infty \subset C_\infty \cap (v+H)_\infty
		= C_\infty \cap H_{\infty} = C \cap H \, .
	\]
\end{proof}

For a fixed lattice $\Lambda \subset V$ (i.e. a finitely generated subgroup), 
a convex polyhedral cone $C \subset V$
is called \emph{rational (with respect to $\Lambda$)}, if each extremal ray of $C$ is generated 
by some element from $C \cap \Lambda$. Note that a face of a rational convex polyhedral cone
is again a rational convex polyhedral cone, see \cite[Proposition~2]{Fu1993Introduction-to-to}.


\begin{lemma}
	\label{lem.Weights_lattice}
	Let $\Lambda \subset V$ be a lattice. If $C \subset V$ is a rational convex
	polyhedral cone (with respect to $\Lambda$), then 
	\[
		C = (C \cap \Lambda)_{\infty} \, .
	\] 
\end{lemma}

\begin{proof}
	Let $v_1, \ldots, v_r \in C \cap \Lambda$ such that $\RR_{\geq 0} v_1, \ldots,\RR_{\geq 0} v_r$
	are the extremal rays of $C$. Then 
	\[
		K:=\Bigset{\sum_{i=1}^r t_i v_i \in V}{\textrm{$0 \leq t_i \leq 1$ for $i=1, \ldots, r$}}
	\]
	is a compact subset of $C$. In particular, there is a real number $\delta \geq 0$ such that
	$\norm{v} \leq \delta$ for all $v \in K$. Now, let $c \in C$. Then there exist
	$m_1, \ldots, m_r \in \ZZ_{\geq 0}$ and $0 \leq t_1, \ldots, t_r \leq 1$ such that
	\[
		c = \underbrace{\sum_{i=1}^r m_i v_i}_{\in C \cap \Lambda} + 
		\underbrace{\left(\sum_{i=1}^r t_i v_i \right)}_{\in K} \, .
	\]
	This shows that $c$ is contained in the $\delta$-neighbourhood $(C \cap \Lambda)^\delta$. In 
	summary we get $C \cap \Lambda \subset C \subset (C \cap \Lambda)^\delta$ and by
	using Lemma~\ref{lem.Properties_asymptotic_cone} and~\ref{lem.Asymptotic_cone_delt_neighb}
	the statement follows.
\end{proof}

\begin{proposition}
	\label{prop.Weights_facet}
	Let $\Lambda \subset V$ be a lattice, let $C \subset V$ be a convex polyhedral cone
	and let $H \subset V$ be a hyperplane such that $H \cap \Lambda$ has rank $\dim H$ and 
	$C \cap H$ is rational 
	with respect to $\Lambda$. Moreover, let $H' := \gamma + H$ for some 
	$\gamma \in \Lambda \setminus H$.
	\begin{enumerate}[leftmargin=*]
		\item 
		\label{prop.Weights_facet1}
		If $C \cap H' \neq \varnothing$ and $C \cap H$ has dimension $\dim H$, then 
		$\inter(C) \cap H' \cap \Lambda \neq \varnothing$.
		\item 
		\label{prop.Weights_facet2}
		If $\inter(C) \cap H' \cap \Lambda$ is non-empty, then 
		\[
			C \cap H = (\inter(C) \cap H' \cap \Lambda)_{\infty} \, .
		\]
	\end{enumerate}
\end{proposition}

\begin{proof}
	\eqref{prop.Weights_facet1}
	As $\gamma \not\in H$, we get $0 \not\in H'$.
	Since $C \cap H' \neq \varnothing$ there is thus $x \in \inter(C) \cap H'$
	by Lemma~\ref{lem.Intersection_with_affine_hyperplane}. 
	As $\dim(C \cap H) = \dim H$, the linear
	span of $C \cap H$ is $H$ and we get that
	$\inter(C \cap H)$ is a non-empty open subset of $H$. Let
	\[
		D := x + (\inter(C\cap H) \setminus \{0\}) \subset \inter(C) \cap H' \setminus \{ x \} \, .
	\]
	Denote by $S$ the unit sphere in $H$ with centre $0$ and consider
	\[
		\pi \colon H' \setminus \{x\} \to S \, , \quad w \mapsto \frac{w-x}{\norm{w-x}} \, .
	\]
	For all $h \in \inter(C \cap H)$ we have $\RR_{> 0}h \subset \inter(C \cap H)$ and thus 
	$\pi^{-1}(\pi(D)) = D$. 
	Since $D$ is a non-empty open subset of $H' \setminus \{x\}$, 
	the same is true
	for $\pi(D)$ in $S$. Since 
	$\gamma \in \Lambda$ and $H' = \gamma + H$, it follows that 
	$H' \cap \Lambda = \gamma + (H \cap \Lambda)$. Since the rank of $H \cap \Lambda$
	is $\dim H$, we get that
	$\pi((H' \cap \Lambda) \setminus \{ x\})$ is dense in $S$. Since $\pi(D)$ is open in $S$, 
	there is $\lambda \in (H' \cap \Lambda) \setminus \{x\} \subset \Lambda$
	with $\pi(\lambda) \in \pi(D)$. 
	In particular, $\lambda \in \pi^{-1}(\pi(D)) = D$ and thus
	$\lambda \in D \cap \Lambda \subset \inter(C) \cap H' \cap \Lambda$.
	
	\eqref{prop.Weights_facet2} By assumption, there is $y \in \inter(C) \cap H' \cap \Lambda$. Thus
	we get
	\[
		y + (C \cap H \cap \Lambda) \subset \inter(C) \cap H' \cap \Lambda \, .
	\]
	This implies by Lemma~\ref{lem.Properties_asymptotic_cone}
	\[
			\label{Eq:1}
			\tag{$a$}
			(C \cap H \cap \Lambda)_{\infty} \subset (\inter(C) \cap H' \cap \Lambda)_{\infty}
			\subset (C \cap H')_{\infty} \, .
	\]
	By Lemma~\ref{lem.Weights_lattice} applied to the rational convex polyhedral cone
	$C \cap H \subset V$ we get
	\[
			\label{Eq:2}
			\tag{$b$}
			C \cap H = (C \cap H \cap \Lambda)_{\infty} \, .
	\] 
	By Lemma~\ref{lem.Weights_facet_translated} we get
	\[
			\label{Eq:3}
			\tag{$c$}
			(C \cap H')_{\infty} = C \cap H \, .
	\]
	Combining~\eqref{Eq:1},~\eqref{Eq:2} and~\eqref{Eq:3} yields the result.
\end{proof}

\begin{proposition}
	\label{prop.help_lemma_weights}
	Let $\Lambda \subset V$ be a lattice of rank $\dim V$, $C \subset V$ a rational 
	convex poly\-hedral cone and $H_0 \subset V$ a hyperplane such that
	$C \cap H_0$ is rational. Let $H_1 \subset V$ be an affine hyperplane which is parallel to $H_0$ and
	set $H_{-1} := -H_1$. 
	If $C \cap H_i \neq \varnothing$ for each $i \in \{\pm1 \}$, then
	\[
		C \cap H_{-1} \cap \Lambda \neq \varnothing \quad \iff \quad
		C \cap H_1 \cap \Lambda \neq \varnothing \, .		
	\]
\end{proposition}

\begin{proof}
	If $H_0 = H_1$, then $H_0 = H_{-1}$ and the statement is trivial. Thus we assume that $H_0 \neq H_1$,
	whence $H_0 \neq H_{-1}$ and $H_{1} \neq H_{-1}$.
		
	Since $C \cap H_{\pm 1} \neq \varnothing$ and since $H_0$, $H_1$ and $H_{-1}$ are pairwise disjoint, 
	there exist $c_{\pm 1} \in \inter(C) \cap H_{\pm 1}$ by Lemma~\ref{lem.Intersection_with_affine_hyperplane}. As $C$ is convex, the line segment in 
	$V$ that connects $c_1$ and $c_{-1}$ lies in $\inter(C)$ and thus 
	$\inter(C) \cap H_0 \neq \varnothing$. Let 
	$B \subset C$ be the union of the proper faces of $C$, i.e. $B$ is the topological
	boundary of $C$ inside the linear span of $C$, see \cite[Sect.~1.2~(7)]{Fu1993Introduction-to-to}.
	If $C \cap H_0 \cap \Lambda \subset B$, then by Lemma~\ref{lem.Weights_lattice} applied to
	the rational convex polyhedral cone 
	$C \cap H_0$ in $V$ we get $C \cap H_0 = (C \cap H_0 \cap \Lambda)_{\infty} \subset B_{\infty} = B$, 
	a contradiction to $\inter(C) \cap H_0 \neq \varnothing$.
	In particular, we can choose 
	\[
		\gamma_0 \in (C \cap H_0 \cap \Lambda) \setminus B \, .
	\]
	
	By exchanging $H_1$ and $H_{-1}$, it is enough to prove
	``$\Rightarrow$" of the statement. 
	For this, let $\gamma_{-1} \in  C \cap H_{-1} \cap \Lambda$. 
	Since $C$ is a convex rational polyhedral cone,
	there  is a finite set $E \subset V \setminus \{0\}$ with
	\[
		C = \bigcap_{u \in E} \set{v \in V}{\sprod{u}{v}\geq 0} \, .
	\]
	Since $\gamma_0 \in C \setminus B$, we get $\sprod{u}{\gamma_0} > 0$ for all $u \in E$. In particular, we can
	choose $m \geq 0$ so big that 
	\[
		\sprod{u}{m \gamma_0 - \gamma_{-1}} = m \sprod{u}{\gamma_0} - \sprod{u}{\gamma_{-1}} \geq 0
	\]
	for all $u \in E$, i.e. $m \gamma_0 - \gamma_{-1} \in C$.
	As $\gamma_0 \in H_0 \cap \Lambda$, we get 
	$m \gamma_0 - \gamma_{-1} \in C \cap H_1 \cap \Lambda$.
\end{proof}

\section{Quasi-affine varieties}

To any variety $X$, we can naturally associate an affine scheme
\[
	X_{\aff} := \Spec \OO(X) \, .
\]
Moreover this scheme comes equipped with the so-called canonical morphism
\[
	\iota \colon X \to X_{\aff}
\]
which is induced by the  natural isomorphism $\OO(X) = \OO(X_{\aff})$.

\begin{remark}
	\label{rem.canonical_morphi_dominant}
	For any variety $X$, the
	canonical morphism $\iota \colon X \to X_{\aff}$ is dominant. Indeed, let
	$X' := \overline{\iota(X)} \subset X_{\aff}$ be the closure of the image of $\iota$
	(endowed with the induced reduced subscheme structure). 
	Since the composition
	\[
		\OO(X) = \OO(X_{\aff}) \to \OO(X') \to \OO(X)
	\]
	is the identity on $\OO(X)$, it follows that the surjection $\OO(X)= \OO(X_{\aff}) \to \OO(X')$
	is injective and thus $X' = X$.
\end{remark}

\begin{lemma}[{\cite[Ch. II, Proposition 5.1.2]{Gr1961Elements-de-geomet-II}}]
	\label{lem:Quasi-affine_irreducible}
	Let $X$ be a variety. Then $X$ is quasi-affine if and only if the canonical morphism
	$\iota \colon X \to X_{\aff}$ is an open immersion.
\end{lemma}


If $X$ is endowed with an algebraic group action, then this
action uniquely extends to an algebraic group action on $X_{\aff}$:

\begin{lemma}
	\label{lem:Quasi-affine_group_action}
	Let $X$ be a quasi-affine $H$-variety for some algebraic group $H$.
	Then $X_{\aff}$ is an affine scheme that has a unique
	$H$-action that extends the $H$-action on $X$ via the canonical open immersion
	$X \hookrightarrow X_{\aff}$.
\end{lemma}

\begin{proof}
	By Lemma~\ref{lem:Quasi-affine_irreducible}, 
	the canonical morphism $X \to X_{\aff}$
	is an open immersion of schemes 
	and there is a unique action of $H$ on $X_{\aff}$ that extends the $H$-action on $X$, see e.g. \cite[Lemma~5]{KrReSa2018Is-the-affine-spac}.
\end{proof}

Now, we compare $G$-sphericity of $X$ and $X_{\aff}$.

\begin{lemma}
	\label{lem:Spherical_quasi-affine}
	Let $G$ be a connected reductive algebraic group and
	let $X$ is a quasi-affine $G$-variety. Then 
	\[
		\textrm{$X$ is $G$-spherical} \quad \iff \quad 
		\textrm{$X_{\aff}$ is an affine $G$-spherical variety} \, .
	\]
\end{lemma}

\begin{proof}	
	%
	If $X_{\aff}$ is an affine $G$-spherical variety, then $X$ is $G$-spherical by Lemma~\ref{lem:Quasi-affine_irreducible}.
	
	For the other implication,
	assume that $X$ is $G$-spherical. It follows that $\OO(X)$ is a 
	finitely generated algebra over the ground field
	by \cite{Kn1993Uber-Hilberts-vier} and thus $X_{\aff} = \Spec \OO(X)$ is an affine variety.
	Since $X$ is irreducible, $X_{\aff}$ is irreducible by Remark~\ref{rem.canonical_morphi_dominant}.
	Moreover, for each $x \in X$, the local ring $\OO_{X, x}$ is integrally
	closed and thus $\OO(X) = \bigcap_{x \in X} \OO_{X, x}$ is integrally closed, i.e.
	$X_{\aff}$ is normal. Since $X$ is an open subset of $X_{\aff}$,
	and since a Borel subgroup of $G$ acts with an open orbit on $X$, the same is true
	for $X_{\aff}$. 
\end{proof}


\begin{proposition}
	\label{prop.locally_finite_and_rational}
	Let $X$ be any variety endowed with an $H$-action for some
	algebraic group $H$. The natural action of $H$ on $\OO(X)$ satisfies the following:
	If $f \in \OO(X)$, then $\Span_{\kk}(Hf)$ is a finite dimensional $H$-invariant subspace 
	of $\OO(X)$ and $H$ acts regularly on it.
\end{proposition}

For the sake of completeness,
we insert a proof here which follows closely \cite[\S2.4 Lemma]{Kr1984Geometrische-Metho}:

\begin{proof}
	The action morphism $\rho \colon H \times X \to X$ induces a $\kk$-algebra homomorphism
	\[
		\rho^{\ast} \colon \OO(X) \to \OO(H \times X) = \OO(H) \otimes \OO(X)
	\]
	(the equality 
	$\OO(H \times X) = \OO(H) \otimes \OO(X)$ follows from \cite[I,\S2,2.6 Proposition]{DeGa1970Groupes-algebrique}).
	Let $f \in \OO(X)$. There exist finitely many linearly independent $f_1, \ldots, f_n \in \OO(X)$
	such that 
	\[
			\rho^{\ast}(f) = \sum_{i=1}^n p_i \otimes f_i
	\]
	for some $p_1, \ldots, p_n \in \OO(H)$. Let $W := \Span_{\kk}(Hf) \subset \OO(X)$.
	Thus $W$ is $H$-invariant.
	We claim that $W$
	is contained in $\bigoplus_{i=1}^n \kk f_i$ (and thus is finite dimensional) and that
	$H$ acts regularly on $W$. 
	Indeed, for all $h \in H$ and $x \in X$ we get
	\[
			\label{eq.locally_finite_and_rational}
			\tag{$\ast$}
			(hf)(x) = \rho^\ast(f)(h^{-1}, x) = \sum_{i=1}^n p_i(h^{-1}) f_i(x) 
	\]
	and thus $hf \in \bigoplus_{i=1}^n \kk f_i$, which proves $W \subset \bigoplus_{i=1}^n \kk f_i$. 
	After changing the basis $f_1, \ldots, f_n$, we may have assumed from the beginning (and we will 
	assume this now) that $W = \bigoplus_{i=1}^m \kk f_i$
	for some $m \leq n$. Using that~\eqref{eq.locally_finite_and_rational} holds for all $x \in X$, 
	we get $p_i(h^{-1}) = 0$
	for each $i > m$ and for each $h \in H$, i.e. $p_i = 0$ for each $i > m$. 
	Thus we may assume that $m = n$
	and $f_1, \ldots, f_n$ is a basis of $W$. For $j \in \{1, \ldots, n\}$, writing 
	$f_j = \sum_r \lambda_{jr} (h_r f)$ for some  $h_r \in H$ and $\lambda_{jr} \in \kk$ yields
	\[
		h f_j = \sum_r \lambda_{jr} (h h_r f) = 
		\sum_i \left(\sum_r \lambda_{jr} p_i(h_r^{-1} h^{-1}) \right) f_i
	\]
	by using~\eqref{eq.locally_finite_and_rational}. Thus $H$ acts regularly on $W$.
\end{proof}

\begin{proposition}
	\label{prop:H-semi-invariants}
	Let $H$ be a connected solvable algebraic group and let $X$ be an irreducible
	quasi-affine $H$-variety.
	Then for every $H$-semi-invariant rational map $f \colon X \dasharrow \kk$
	there exist $H$-semi invariants $f_1, f_2 \in \OO(X)$ such
	that $f = f_1 / f_2$.
\end{proposition}

\begin{proof}
	The strategy of the proof is the same as in \cite[Proposition 2.8]{Br2010Introduction-to-ac}.
	Let $f \colon X \dasharrow \kk$ be a $H$-semi-invariant rational map. The subspace
	\[
			V = \set{g \in \OO(X)}{g f \in \OO(X)}	\subset \OO(X)	
	\]
	is non-zero, since the quotient field of $\OO(X)$ is the function field $K(X)$
	(note that $X$ is quasi-affine). Moreover, $V$ is $H$-invariant, since $f$
	is a $H$-semi-invariant. By Proposition~\ref{prop.locally_finite_and_rational} there
	is a finite-dimensional $H$-invariant non-zero subspace $W \subset V$ such
	that the $H$-action on $W$ is regular. Then by the \name{Lie-Kolchin} 
	Theorem~\cite[Theorem~17.6]{Hu1975Linear-algebraic-g} 
	there is a non-zero $H$-semi-invariant $f_2 \in W$.
	Thus $f_1 = f_2 \cdot f \in \OO(X)$ is $H$-semi-invariant which concludes the proof.
\end{proof}

\begin{corollary}
	Assume that $H$ is a unipotent algebraic group and $X$ is an irreducible 
	quasi-affine $H$-variety.
	Then the field $K(X)^H$ of $H$-invariant rational maps $X \dasharrow \kk$
	is the same as the quotient field of $\OO(X)^H$.
\end{corollary}

\begin{proof}
	There is an inclusion of the quotient field $Q(\OO(X)^H)$ into $K(X)^H$.
	As the character group of $H$ is trivial, Proposition~\ref{prop:H-semi-invariants} implies $Q(\OO(X)^H) = K(X)^H$.
\end{proof}

%

\section{Vector fields}

\subsection{Generalities on vector fields}
\label{sec.generalities_on_vectorfields}
Let $X$ be any variety. We denote by
$\Ve(X)$ the vector space of all 
algebraic vector fields on $X$, i.e. all algebraic sections of the tangent bundle
$T X \to X$. Note that $\Ve(X)$ is in a natural way an $\OO(X)$-module.

Now, assume $X$ is endowed with a regular action of an algebraic
group $H$. Then, $\Ve(X)$ is an $H$-module, via the following action: Let $h \in H$ and
$\xi \in \Ve(X)$, then $h \cdot \xi$ is defined via
\[
	(h \cdot \xi) (x) = \mathrm{d} \varphi_h (\xi(\varphi_{h^{-1}}(x))) \quad \textrm{for each $x \in X$}
\]
where $\varphi_h$ denotes the automorphism of $X$ given by
multiplication with $h$ and $d\varphi_h$ denotes the differential of $\varphi_h$.
For a fixed character $\lambda \colon H \to \GG_m$ we say that a
vector field $\xi \in \Ve(X)$ is
\emph{normalized by $H$ with weight $\lambda$} if
$\xi$ is a $H$-semi-invariant of weight $\lambda$, i.e.
for all $h \in H$ the following diagram commutes
\[
\xymatrix{
	TX  \ar[rr]^{ \textrm{d} \varphi_h} && TX \\
	X \ar[u]^{\xi} \ar[rr]^{\varphi_h} && X \ar[u]_{\lambda(h) \xi} 
}
\]
We denote by $\Ve(X)_{\lambda, H}$
the subspace in $\Ve(X)$ of all vector fields which are normalized by $H$
with weight $\lambda$. If it is clear which action on $X$ is meant, we drop
the index $H$ and simply write $\Ve(X)_{\lambda}$. Note that $\Ve(X)_{\lambda}$
is in a natural way an $\OO(X)^H$-module where $\OO(X)^H$ 
denotes the $H$-invariant regular functions on $X$. 
We denote by $\Ve^H(X)$ the subspace of all 
$H$-invariant vector fields in $\Ve(X)$, i.e. $\Ve^H(X) = \Ve(X)_{0}$ where $0$ 
denotes the trivial character of $H$.

Now, assume that $X$ is affine. There is a $\kk$-linear map
\[
	\Ve(X) \to \Der_{\kk}(\OO(X)) \, , \quad \xi \mapsto D_{\xi}
\] 
where $D_{\xi} \colon \OO(X) \to \OO(X)$ is given by $D_{\xi}(f)(x) := \xi(x)(f)$
(here we identify the tangent space of $X$ at $x$ with the 
$\kk$-derivations $\OO_{X, x}\to \kk$ in $x$). In fact, $\Ve(X) \to \Der_{\kk}(\OO(X))$ is an 
isomorphism: Indeed, as $X$ is affine, we have
\[
	\Ve(X) = \Bigset{ \eta \colon X \to TX }{
		\begin{array}{l}
		\textrm{$\eta$ is a set-theoretical section and for all $f \in \OO(X)$} \\
		\textrm{the map $x \mapsto \eta(x)(f)$ is a regular function on $X$}
		\end{array}
	} \, ,
\]
see \cite[\S3.2]{FuKr2018On-the-geometry-of}.

\subsection{Homogeneous $\GG_a$-actions and vector fields}
\label{sec.GroupActionsAndVectorfields}

The material of this small subsection is contained in \cite[\S6.5]{FuKr2018On-the-geometry-of},
however formulated for all varieties.
Throughout this subsection, $H$ denotes an algebraic group and $X$ a $H$-variety. 

Let $D$ be an algebraic group that acts regularly on a variety $X$.
Then we get a $\kk$-linear map $\Lie D \to \Ve(X)$, 
$A \mapsto \xi_A$, where the vector field $\xi_A$ is given by
\[
	\label{Eq:vectorfield_to_H-action}
	\tag{$\varominus$}
	\xi_A \colon X \to TX \, , \quad x \mapsto (\textrm{d}_e \mu_x) A
\]
and $\mu_x \colon D \to X$, $d \mapsto d x$ denotes the orbit map in $x$: Indeed~\eqref{Eq:vectorfield_to_H-action} is a morphism as it is the composition of the morphisms
\[
	X \to T_e D \times TX \, , \ x \mapsto (A, 0_x) \quad \textrm{and} \quad
	d \mu |_{T_e D \times TX} \colon T_e D \times TX \to TX
\]
where $0_x \in TX$ denotes the zero vector inside $T_x X$ and 
$\mu \colon D \times X \to X$ denotes the $D$-action.

\begin{lemma}
	\label{lem.faithfull_action}
	Let $D$ be an algebraic group. If $D$ acts faithfully on a variety $X$, then the
	$\kk$-linear map $\Lie D \to \Ve(X)$, $A \mapsto \xi_A$ is injective.
\end{lemma}

\begin{proof}
	For each $x \in X$, the kernel of the differential $\textrm{d}_e \mu_x \colon \Lie D \to T_x X$ of
	the orbit map $\mu_x \colon D \to X$, $d \mapsto d x$ is equal to $\Lie D_x$ where
	$D_x$ denotes the stabilizer of $x$ in $D$. If $A \in \Lie D$ such that $\xi_A = 0$, then 
	$(\textrm{d}_e \mu_x)A = 0$ for each $x \in X$, i.e. $A \in \Lie D_x$ for each $x \in X$.
	As $D$ acts faithfully on $X$, we get $\{ e \} = \bigcap_{x \in X} D_x$ and thus
	$\{0\} = \Lie( \bigcap_{x \in X} D_x ) = \bigcap_{x \in X} \Lie(D_ x)$ which implies $A = 0$.
\end{proof}

A $\GG_a$-action on $X$ is called 
\emph{$H$-homo\-geneous of weight $\lambda \in \frak{X}(H)$} if 
\[
	h \circ \varepsilon(t) \circ h^{-1} = \varepsilon(\lambda(h) \cdot t) \quad 
	\textrm{for all $h \in H$ and all $t \in \GG_a$}
\]
where $\varepsilon \colon \GG_a \to \Aut(X)$ is the group homomorphism
induced by the $\GG_a$-action. 

\begin{lemma}
	\label{lem.H-homogeneous_Ga-action}
	Let $H$ be an algebraic group, $X$ a $H$-variety and
	$\rho$ a $H$-homogeneous $\GG_a$-action on $X$
	of weight $\lambda \in \frak{X}(H)$. Then the image
	of the previously introduced $\kk$-linear map 
	$\Lie \GG_a \to \Ve(X)$, $A \mapsto \xi_A$ associated to $\rho$
	lies in $\Ve(X)_{\lambda, H}$.
\end{lemma}

\begin{proof}
	As $\rho$ is $H$-homogeneous, we get for each $x \in X$ and each $h \in H$ 
	the following commutative diagram
	\[
		\xymatrix{
			\GG_a \ar[d]_{\mu_x} \ar[rr]^-{t \mapsto \lambda(h) t} && \GG_a \ar[d]^-{\mu_{hx}} \\
			X \ar[rr]^-{\varphi_h} && X 
		}	
	\]
	where $\varphi_h \colon X \to X$ denotes multiplication by $h$. Taking differentials in the
	neutral element $e \in \GG_a$
	gives $\textrm{d}_x \varphi_h \textrm{d}_e \mu_x = \lambda(h) \textrm{d}_e \mu_{hx}$
	for each $A \in \Lie \GG_a$. This implies that $h \cdot \xi_A(x) = \lambda(h) \xi_A(x)$
	for each $A \in \Lie \GG_a$ and thus the statement follows.
\end{proof}

\begin{lemma}
	\label{lem:D(X)_contained_in_the_weights_of_Ve^H(X)}
	Let $H$ be an algebraic group and let $N \subset H$ 
	be a normal subgroup such that the character group $\frak{X}(N)$ is trivial. 
	If $X$ is an irreducible $H$-variety, then
	\[
		D_H(X) = 
		\Bigset{\lambda \in \frak{X}(H)}{
			\begin{array}{l}
				\textrm{there is a non-trivial $H$-homogeneous} \\
				\textrm{$\GG_a$-action on $X$ of weight $\lambda$}
			\end{array}
		}
	\]
	is contained in the set of $H$-weights of non-zero vector fields in $\Ve^{N}(X)$ that are normalized
	by $H$.
\end{lemma}

\begin{proof}
	Let $\rho \colon \GG_a \times X \to X$ be a non-trivial $\GG_a$-action on $X$.
	By Lemma~\ref{lem.faithfull_action} and~\ref{lem.H-homogeneous_Ga-action} 
	there is a non-zero $\xi \in \Ve(X)$ such that 
	for each $h \in H$ we have $h \cdot \xi = \lambda(h) \xi$.
	Moreover, since $\frak{X}(N) = 0$, $\xi$ is $N$-invariant.
	Thus $D_H(X)$ is contained in the set of $H$-weights of non-zero vector fields in
	$\Ve^{N}(X)$ that are normalized by $H$.
\end{proof}

Now, assume that $X$ is an affine variety and fix some non-zero element $A_0 \in \Lie \GG_a$.
Moreover, denote by $\LND_{\kk}(\OO(X)) \subset \Der_\kk(\OO(X))$ the cone of locally nilpotent derivations
on $\OO(X)$, i.e. the cone in $\Der_\kk(\OO(X))$ of $\kk$-derivations $D$ of $\OO(X)$ such that
for all $f \in \OO(X)$ there is a $n = n(f) \geq 1$ such that $D^n(f) = 0$, where $D^n$ denotes the $n$-fold 
composition of $D$. There is a map
\[
	\{ \, \textrm{$\GG_a$-actions on $X$} \, \} \stackrel{1:1}{\longleftrightarrow} \LND_{\kk}(\OO(X)) \, , \quad \rho \mapsto D_{\xi_{A_0}}
\]
where $\xi_{A_0}$ is defined as in~\eqref{Eq:vectorfield_to_H-action} with respect to the $\GG_a$-action $\rho$. As for each $f \in \OO(X)$ we have that $D_{\xi_{A_0}}(f)$ is the morphism 
$x \mapsto A_0(f \circ \mu_x)$ (we interpret $A_0$ as a $\kk$-derivation of $\OO_{\GG_a, e} \to \kk$ in $e$),
it follows from \cite[Sect. 1.5]{Fr2017Algebraic-theory-o} that the above map is in fact a bijection.

\subsection{Finiteness results on modules of vector fields}

Let $G$ be an algebraic group. For this subsection, let $X$ be a  $G$-variety.
Note that $\Ve(X)$ is an 
\emph{$\OO(X)$-$G$-module} via the $\OO(X)$- and $G$-module structures given in \S\ref{sec.generalities_on_vectorfields}, 
i.e. $\Ve(X)$ is a $G$-module, it is an $\OO(X)$-module 
and both structures are compatible in the sense that 	 
\[
	g \cdot (f \cdot \xi) = (g \cdot f) \cdot (g \cdot \xi) \quad
	\textrm{for all $g \in G$, $f \in \OO(X)$ and $\xi \in \Ve(X)$.}
\]

\begin{lemma}
	\label{lem.finiteness_and_rationality_of_VecX}
	Assume that $X$ is a quasi-affine $G$-variety and that $\OO(X)$ is finitely generated as a $\kk$-algebra.
	Then the $\OO(X)$-$G$-module $\Ve(X)$ is finitely generated and rational, i.e.
	$\Ve(X)$ is finitely generated as an $\OO(X)$-module 
	and the $G$-representation $\Ve(X)$ is a sum of finite dimensional 
	rational $G$-subrepresentations.
\end{lemma}

\begin{proof}
	Since $\OO(X)$ is finitely generated, $X_{\aff} = \Spec \OO(X)$ is an affine variety
	that is endowed with a natural $G$-action, see Lemma~\ref{lem:Quasi-affine_group_action}. By
	\cite[Satz 2, II.2.S]{Kr1984Geometrische-Metho} there is a rational $G$-representation $V$ and a $G$-equivariant 
	closed embedding  $X_{\aff} \subseteq V$. We denote by
	\[
		\iota \colon X \to V
	\]
	the composition of the canonical open immersion $X \subset X_{\aff}$ with $X_{\aff} \subset V$.
	Note that the image of $\iota$ is locally closed in $V$ and that $\iota$ induces
	an isomorphism of $X$ onto that locally closed subset of $V$.
	Thus,  $\mathrm{d} \iota \colon T X \to T V |_X$ is a $G$-equivariant closed 
	embedding over $X$ which is linear on each fibre of $T X \to X$.
	Thus we get an $\OO(X)$-$G$-module embedding 
	\[
	\Ve(X) \to \Gamma(TV |_X) \, , \quad
	\xi \mapsto \mathrm{d} \iota \circ \xi,
	\]
	where $\Gamma(TV |_X)$ denotes the $\OO(X)$-$G$-module
	of sections of $TV |_X \to X$. However, since the vector bundle
	$TV |_X \to X$ is trivial, there is a 
	$\OO(X)$-$G$-module isomorphism
	\[
	\Gamma(TV |_X) \simeq \Mor(X, V), 
	\]
	where $G$ acts on $\Mor(X, V)$ via $g \cdot \eta = (x \mapsto g\eta(g^{-1}x))$.
	Now, the $\OO(X)$-$G$-module 
	$\Mor(X, V) \simeq \OO(X) \otimes_{\kk} V$ is finitely generated and 
	rational (see Proposition~\ref{prop.locally_finite_and_rational}), 
	and thus the statement follows.
\end{proof}

For the next result we recall the following definition.

\begin{definition}
	Let $G$ be an algebraic group. A closed subgroup $H \subset G$ 
	is called a \emph{Grosshans subgroup} if $G/H$ is quasi-affine and 
	$\OO(G/H) = \OO(G)^H$ is a finitely generated $\kk$-algebra.
\end{definition}

Let $G$ be a connected reductive algebraic group.
Examples of Grosshans subgroups of $G$ are unipotent radicals of parabolic
subgroups of $G$, see \cite[Theorem~16.4]{Gr1997Algebraic-homogene}. 
In particular, the unipotent radical $U$ of a Borel subgroup
$B \subset G$ is a Grosshans subgroup in $G$
(see also \cite[Theorem~9.4]{Gr1997Algebraic-homogene}). A very important property of Grosshans subgroups is the following:

\begin{proposition}[{\cite[Theorem~9.3]{Gr1997Algebraic-homogene}}]
	\label{prop:finite_generation_of_invariant}
	Let $A$ be a finitely generated $\kk$-algebra and let $G$
	be a connected reductive algebraic group that acts via $\kk$-algebra automorphisms on $A$
	such that $A$ becomes a rational $G$-module.
	If $H \subset G$ is a Grosshans subgroup, then the ring of $H$-invariants
	\[
		A^H = \set{a \in A}{\textrm{$h a = a$ for all $h \in H$}}
	\]
	is a finitely generated $\kk$-subalgebra of $A$.
\end{proposition}

\begin{proposition}
	\label{prop.MU_finitely_generated}
	Let $R$ be a finitely generated $\kk$-algebra and assume
	that a connected reductive algebraic group 
	$G$ acts via $\kk$-algebra automorphisms on $R$ such that
	$R$ becomes a rational $G$-module. Let $H \subset G$ be a 
	Grosshans subgroup.
	If $M$ is a finitely generated rational $R$-$G$-module, 
	then $M^H$ is a finitely generated $R^H$-module.
\end{proposition}

\begin{proof}
	We consider the $\kk$-algebra $A = R \oplus \varepsilon M$, where the 
	multiplication on $A$ is defined via
	\[
	(r + \varepsilon m) \cdot (q + \varepsilon n) =
	rq + \varepsilon (rn+qm) \, . 
	\]
	Since $R$ is a finitely generated $\kk$-algebra and since $M$
	is a finitely generated $R$-module, $A$ is a finitely generated $\kk$-algebra.
	Moreover, since $R$ and $M$ are rational $G$-modules, 
	$A$ is a rational $G$-module. Moreover, 
	$G$ acts via $\kk$-algebra automorphisms on $A$.
	Since $H$ is a Grosshans subgroup of $G$, it follows now
	by Proposition~\ref{prop:finite_generation_of_invariant}	
	that 
	\[
	A^H = R^H \oplus \varepsilon M^H
	\] 
	is a finitely generated $\kk$-algebra. Thus one can choose 
	finitely many elements $m_1, \ldots, m_k \in M^H$ 
	such that $\varepsilon m_1, \ldots, \varepsilon m_k$ generate $A^H$
	as an $R^H$-algebra. However, since $\varepsilon^2 = 0$, it follows
	that $m_1, \ldots, m_k$ generate $M^H$ as an $R^H$-module.
	
	As $M$ is a rational $G$-module, it follows that $M^H$ is a rational $H$-module.
\end{proof}

As an application of Lemma~\ref{lem.finiteness_and_rationality_of_VecX} and
Proposition~\ref{prop.MU_finitely_generated} we get the following 
finiteness result of $\Ve^{H}(X)$ for a Grosshans subgroup $H$ of a connected reductive
algebraic group.

\begin{corollary}
	\label{cor.VecUX_finitely_generated}
	Let $H$ be a Grosshans subgroup of a connected reductive algebraic group 
	$G$. If $X$ is a quasi-affine $G$-variety
	such that $\OO(X)$ is finitely generated as a $\kk$-algebra, then
	$\Ve^{H}(X)$ is a finitely generated $\OO(X)^{H}$-module.
\end{corollary}

\subsection{Vector fields normalized by a group action with an open orbit}

For this subsection, let $H$ be an algebraic group and let $X$ be an irreducible 
$H$-variety which contains an open $H$-orbit. 
Moreover, fix a character $\lambda$ of $H$. 
In this section we give an upper bound on the dimension of 
$\Ve(X)_\lambda = \Ve(X)_{\lambda, H}$. 

\begin{lemma}
	\label{lem:Normalized_Derivations}
	Fix $x_0 \in X$ that lies in the open $H$-orbit and let $H_{x_0}$ be 
	the stabilizer of $x_0$ in $H$. Then, there exists an injection
	of $\Ve(X)_\lambda$ into the $H_{x_0}$-eigenspace of the tangent space $T_{x_0} X$ 
	of weight $\lambda$ given by
	\[
		\xi \mapsto \xi(x_0) \, .
	\]
	In particular, the dimension of $\Ve(X)_\lambda$ is smaller than or equal to the
	dimension of the
	$H_{x_0}$-eigenspace of weight $\lambda$ of $T_{x_0} X$.
\end{lemma}

\begin{proof}
	Let $\xi \in \Ve(X)_\lambda$. By definition we have for all $h \in H$
	\[
		\label{eq:equivariancy}
		\tag{$\varobar$}
		\lambda(h) \xi(hx_0) = \textrm{d} \varphi_h \xi(x_0)
	\]
	where $\varphi_h \colon X \to X$ denotes the automorphism given by
	multiplication with $h$. Since $x_0$ lies in the open $H$-orbit and $X$ is irreducible,
	$\xi$ is uniquely determined by $\xi(x_0)$. Moreover,~\eqref{eq:equivariancy} implies that $\xi(x_0)$ is a
	$H_{x_0}$-eigenvector of weight $\lambda$ of $T_{x_0} X$.
\end{proof}

\section{Automorphism group of a variety and root subgroups}
\label{sec.AutomorphismGroups}

Let $X$ be a variety and denote by $\Aut(X)$ its automorphism group. 
A subgroup $H \subset \Aut(X)$ is called \emph{algebraic subgroup of $\Aut(X)$} if 
$H$ has the structure of an algebraic group such that the
action $H \times X \to X$ is a regular action of the algebraic group $H$ on $X$. It follows from~\cite{Ra1964A-note-on-automorp} (see also \cite[Theorem 2.9]{KrReSa2018Is-the-affine-spac})
that this algebraic group structure on $H$ is unique up to algebraic group isomorphisms.

Let $X$, $Y$ be varieties.
We say that a group homomorphism $\theta \colon \Aut(X) \to \Aut(Y)$ \emph{preserves algebraic subgroups} if
for each algebraic subgroup $H \subset \Aut(X)$ its image $\theta(H)$ is an algebraic subgroup of $\Aut(Y)$
and if the restriction $\theta|_{H} \colon H \to \theta(H)$ is a homomorphism of algebraic groups. 
We say that a group isomorphism $\theta \colon \Aut(X) \to \Aut(Y)$ 
\emph{preserves algebraic subgroups}, if both 
homomorphisms $\theta$ and $\theta^{-1}$ preserve algebraic subgroups.

Assume now that $X$ is an $H$-variety for some algebraic group $H$ and that 
$U_0 \subset \Aut(X)$ is
a one-parameter unipotent subgroup, i.e. an algebraic subgroup of $\Aut(X)$ that
is isomorphic to $\GG_a$. If for some isomorphism $\GG_a \simeq U_0$
of algebraic groups the induced $\GG_a$-action on $X$ is 
$H$-homogeneous of weight $\lambda \in \frak{X}(H)$, then we call 
$U_0$ a \emph{root subgroup with respect to $H$ 
of weight $\lambda$} (see \S\ref{sec.GroupActionsAndVectorfields}). Note that this
definition does not depend on the choice of the isomorphism $\GG_a \simeq U_0$. This notion goes back
to \name{Demazure} \cite{De1970Sous-groupes-algeb}.


\begin{lemma}
	\label{lem.auto_perserves_weights}
	Let $X$, $Y$ be $H$-varieties for some algebraic group $H$. 
	If $\theta \colon \Aut(X) \to \Aut(Y)$ is a group homomorphism that
	preserves algebraic subgroups and if $\theta$ is compatible with the $H$-actions
	in the way that
	\[
		\xymatrix@=10pt{
			& H \ar[ld] \ar[rd] & \\
			\Aut(X) \ar[rr]^-{\theta} && \Aut(Y) \, .
		}
	\]
	commutes, then for any root subgroup $U_0 \subset \Aut(X)$ with respect to
	$H$, the image $\theta(U_0)$ is either the trivial group or
	a root subgroup with respect to $H$ of  the same weight as $U_0$.
\end{lemma}

\begin{proof}
	We can assume that $\theta(U_0)$ is not the trivial group.
	Hence, $\theta(U_0)$ is a one-parameter unipotent group. 
	
	Let $\varepsilon \colon \GG_a \simeq U_0 \subset \Aut(X)$ 
	be an isomorphism and let 
	$\lambda \colon H \to \GG_m$ be the weight of $U_0$. Then we have for each $t \in \GG_a$	
	\[
		h \circ \theta(\varepsilon(t)) \circ h^{-1} =
		\theta(h \circ \varepsilon(t) \circ h^{-1}) = \theta(\varepsilon(\lambda(h)\cdot t)) \, .
	\]
	Since $\theta |_{U_0} \colon U_0 \to \theta(U_0)$ is a surjective 
	homomorphism of algebraic groups that are both isomorphic to $\GG_a$ (and since the ground field is of
	characteristic zero), $\theta |_{U_0}$ is in fact an isomorphism. 
	Thus $\theta \circ \varepsilon \colon \GG_a \simeq \theta(U_0) \subset \Aut(X)$ is an isomorphism
	and hence $\lambda$ is the weight of $\theta(U_0)$ with respect to $H$.
\end{proof}

\section{Homogeneous $\GG_a$-actions on quasi-affine toric varieties}

In this section, we give a description of the homogeneous $\GG_a$-actions
on a quasi-affine toric variety. Throughout this section, we denote by $T$ an algebraic torus.
Recall that a $T$-toric variety is a $T$-spherical variety. 
A $\GG_a$-action is called 
\emph{homogeneous} if it is $T$-homogeneous 
of some weight $\lambda \in \frak{X}(T)$, see~\S\ref{sec.GroupActionsAndVectorfields}.

Let $X$ be a toric variety. In case $X$ is affine, Liendo gave a full description
of all homogeneous $\GG_a$-action. In case $X$ is quasi-affine, $X_{\aff} = \Spec(\OO(X))$ is an affine
$T$-toric variety by Lemma~\ref{lem:Spherical_quasi-affine}. 
Moreover, every homogeneous $\GG_a$-action
on $X$ extends uniquely to a homogeneous $\GG_a$-action on $X_{\aff}$ by Lemma~\ref{lem:Quasi-affine_group_action}.
Thus we are led to the problem of describing the homogeneous $\GG_a$-actions on $X_{\aff}$ that preserve
the open subvariety $X$.

In order to give this description we have to do some preparation.
First, we give a description of $X_{\aff}$ in case $X$ is toric and
give a characterization, when $X$ is quasi-affine. For this, let us introduce some basic terms
from toric geometry. As a reference we take~\cite{Fu1993Introduction-to-to} and~\cite{CoLiSc2011Toric-varieties}.

Note that $M = \frak{X}(T) = \Hom_\ZZ(N, \ZZ)$ for a free abelian group $N$ of rank $\dim T$ 
and denote by $M_{\RR} = M \otimes_{\ZZ} \RR$, $N_{\RR} = N \otimes_{\ZZ} \RR$ the extensions to $\RR$.
Moreover, let
\[
	M_{\RR} \times N_{\RR} \to \RR \, , \quad
	(u, v) \mapsto \langle u, v \rangle
\]
be the canonical $\RR$-bilinear form. Denote by $\kk[M]$ the $\kk$-algebra with basis $\chi^m$ for all $m \in M$
and multiplication $\chi^m \cdot \chi^{m'} = \chi^{m+m'}$. Note that there is an identification
\[
	T = \Spec \kk[M] \, .
\]

Let $\sigma \subset N_{\RR}$ be a \emph{strongly convex rational
polyhedral cone} in $N_{\RR}$, i.e. it is a convex rational polyhedral cone with respect to the lattice $N \subset N_{\RR}$ and $\sigma$ contains no non-zero linear subspace of $N_\RR$. Then its dual
\[
	\sigma^\vee = \set{u \in M_{\RR}}{\textrm{$\langle u, v \rangle \geq 0$ for all $v \in \sigma$}}
\]
is a convex rational polyhedral cone in $M_\RR$.
Denote by $\sigma_M^\vee$ the intersection of $\sigma^\vee$ with $M$ inside $M_{\RR}$.
We can then associate to $\sigma$ a toric variety
\[
	X_\sigma = \Spec \kk[\sigma_M^\vee], \quad \textrm{where} \quad
	\kk[\sigma_M^\vee] = \bigoplus_{m \in \sigma_M^\vee} \kk \chi^m \subset \kk[M] \, .
\]
The torus $T$  acts then on $X_{\sigma}$ with an open orbit where this action is induced by the
coaction 
$\kk[\sigma_M^\vee] \to \kk[\sigma_M^\vee] \otimes_{\kk} \kk[M]$, 
$\chi^u \mapsto \chi^u \otimes \chi^u$. Note that we have a order-reversing 
bijection between the faces of $\sigma$ and the faces of its dual $\sigma^\vee$:
\[
	\{ \, \textrm{faces of $\sigma$} \, \} \stackrel{1:1}{\longleftrightarrow} 
	\{ \, \textrm{faces of $\sigma^\vee$} \, \} \, , \quad
	\tau \mapsto \sigma^\vee \cap \tau^{\bot}
\]
where $\tau^\bot$ consists of those $u \in M_\RR$ that satisfy $\sprod{u}{v} = 0$ for all $v \in \tau$,
see \cite[Sect.~1.2~(10)]{Fu1993Introduction-to-to}. Moreover, each face $\tau \subset \sigma$
determines an orbit of dimension $n - \dim(\tau)$ of the $T$-action on $X_\sigma$ (see \cite[\S3.1]{Fu1993Introduction-to-to}). 
We denote its closure in $X_\sigma$ by $V(\tau)$. In particular, $V(\tau)$ is an irreducible closed
$T$-invariant subset of $X_{\sigma}$.

More generally, for a fan $\Sigma$ of strongly convex rational polyhedral 
cones in $N_{\RR}$ we denote by $X_\Sigma$ its 
associated toric variety, which is covered by the open affine toric subvarieties $X_{\sigma}$
where $\sigma$ runs through the cones in $\Sigma$.

\begin{lemma}
	\label{lem.toric_qusi-affine}
	Let $X = X_{\Sigma}$ be a toric variety for a fan $\Sigma$ of strongly convex rational 
	polyhedral cones in $N_{\RR}$. Denote by $\sigma_1, \ldots, \sigma_r \subset N_{\RR}$
	the maximal cones in $\Sigma$ and set
	\[
		\sigma = \Conv \bigcup_{i=1}^r \sigma_i \subset N_{\RR} \, .
	\]
	Then:
	\begin{enumerate}[leftmargin=*]
		\item \label{lem.toric_qusi-affine1} We have $X_{\aff} = X_\sigma$
		and the canonical morphism $\iota \colon X \to X_{\aff}$
		is induced by the embeddings $\sigma_i \subset \sigma$ for $i=1, \ldots, r$.
		\item \label{lem.toric_qusi-affine2} 
		The toric variety $X$ is quasi-affine if and only if each $\sigma_i$ is a face of $\sigma$.
		\item \label{lem.toric_qusi-affine3}
		If $X$ is quasi-affine, then the irreducible components of $X_{\aff} \setminus X$ are the
		closed sets of the form $V(\tau)$ where $\tau$ is a minimal face of $\sigma$ with $\tau \not\in \Sigma$.
		\item \label{lem.toric_qusi-affine4}
		If $X$ is quasi-affine, then each face $\tau$ of $\sigma$ with $\tau \not\in \Sigma$
		has dimension at least $2$. In particular, 
		$X_{\aff} \setminus X$ is a closed subset of codimension at least $2$
		in $X_{\aff}$.
	\end{enumerate}
\end{lemma}

\begin{proof}[Proof of Lemma~\ref{lem.toric_qusi-affine}]
	\eqref{lem.toric_qusi-affine1}: Since the affine toric varieties $X_{\sigma_1}, \ldots, X_{\sigma_r}$
	cover $X$, we get inside $\OO(T) = \kk[M]$:
	\[
	\OO(X) = \bigcap_{i=1}^r \OO(X_{\sigma_i}) = 
	\bigcap_{i=1}^r \kk[(\sigma_i)_M^{\vee}] = 
	\kk\left[ \left( \bigcap_{i=1}^r \sigma_i^{\vee} \right) \cap M \right] \, .
	\]
	Since
	\[
	\sigma^\vee = \left(\Conv \bigcup_{i=1}^r \sigma_i\right)^\vee =
	\set{u \in M_{\RR}}{\textrm{$\langle u, v \rangle \geq 0$ for all $v \in \sigma_i$ and all $i$}} = 
	\bigcap_{i=1}^r \sigma_i^\vee \, ,
	\]
	we get $\OO(X) = \OO(X_\sigma)$ which implies the first claim. 
	For the second claim, denote by $\iota_i \colon X_{\sigma_i} \to (X_{\sigma_i})_{\aff}$
	the canonical morphism of $X_{\sigma_i}$
	(which is in fact an isomorphism). Then we have for each 
	$i=1, \ldots, r$ the commutative diagram 
	\[
	\xymatrix@R=10pt{
		X \ar[r]^-{\iota} & X_{\aff} \ar@{=}[r] &  X_\sigma \\
		X_{\sigma_i} \ar@{}[u] |-{\bigcup} \ar[r]^-{\iota_i} & (X_{\sigma_i})_{\aff} \ar[u]_-{\eta}
	}
	\]
	where $\eta$ is induced by the inclusion $\kk[\sigma^{\vee}_M] \subset \kk[(\sigma_i)^\vee_M]$.
	As $\eta \circ \iota_i \colon X_{\sigma_i} \to X_{\aff} = X_{\sigma}$
	is induced by the inclusion $\sigma_i \subset \sigma$, the second claim follows.
	
	\eqref{lem.toric_qusi-affine2}: If $\sigma_i \subset \sigma$ is a face, then
	the induced morphism $X_{\sigma_i} \to X_{\sigma}$ is an open immersion
	(see \cite[\S1.3 Lemma]{Fu1993Introduction-to-to}). Now, if each $\sigma_i$
	is a face of $\sigma$, then by~\eqref{lem.toric_qusi-affine1} the canonical morphism
	$\iota \colon X \to X_\sigma$ is an open immersion, i.e. $X$ is quasi-affine (see Lemma~\ref{lem:Quasi-affine_irreducible}).
	
	On the other hand, if $X$ is quasi-affine, then $\iota \colon X \to X_\sigma$ is an open immersion
	(again by Lemma~\ref{lem:Quasi-affine_irreducible}) and by~\eqref{lem.toric_qusi-affine1}, the
	morphism $X_{\sigma_i} \to X_{\sigma}$ induced by $\sigma_i \subset \sigma$ 
	is also an open immersion. It follows now from
	\cite[\S1.3 Exercise on p.18]{Fu1993Introduction-to-to} that $\sigma_i$ is a face of $\sigma$.
	
	\eqref{lem.toric_qusi-affine3}: We claim, that $X_{\aff} \setminus X$ is the union
	of all $V(\tau)$, where $\tau \subset \sigma$ is a face with $\tau \not\in \Sigma$.

	Let $\tau \subset \sigma$ be a face such that 
	$\tau \not\in \Sigma$. In particular we have for all $i$ that $\tau \not\subset \sigma_i$. 
	Sine $X$ is quasi-affine, $\sigma_i$ is a face of $\sigma$ by~\eqref{lem.toric_qusi-affine2}.
	Hence, there is a $u_i \in \sigma^\vee_M$ with $u_i^\bot \cap \sigma = \sigma_i$ and
	\[
		X_{\sigma_i} = X_\sigma \setminus Z_{X_{\sigma}}(\chi^{u_i})
	\]
	by~\cite[\S1.3 Lemma]{Fu1993Introduction-to-to} where $Z_{X_{\sigma}}(\chi^{u_i})$ denotes
	the zero set of $\chi^{u_i} \in \OO(X_{\sigma})$ inside $X_{\sigma}$. 
	As $\tau \subset \sigma$, but $\tau \not\subset \sigma_i$, we get
	$\tau \not\subset u_i^\bot$ and thus $u_i \in \sigma_M^\vee \setminus \tau^\bot$.
	By~\cite[\S3.1]{Fu1993Introduction-to-to}, the 
	closed embedding $V(\tau) \subset X_\sigma$ corresponds to the surjective $\kk$-algebra homomorphism
	\[
		\kk[\sigma^{\vee}_M] \to \kk[\sigma^{\vee}_M \cap \tau^{\bot}] \, , 
		\quad 
		\chi^m \mapsto \left\{\begin{array}{rl}
		\chi^m \, , & \textrm{if $m \in \tau^{\bot}$} \\
		0 \, , & \textrm{if $m \in \sigma_M^\vee \setminus \tau^\bot$}
		\end{array}\right. \, .
	\]
	In particular, $\chi^{u_i}$ vanishes on $V(\tau)$ and thus $V(\tau)$
	and $X_{\sigma_i}$ are disjoint for all $i=1, \ldots, r$, i.e. 
	$V(\tau) \subset X_{\aff} \setminus X$. 
	On the other hand, if $\eta \subset \sigma$ is a face with $\eta \in \Sigma$, then there is a
	$i \in \{1, \ldots, r \}$ such that $\eta$ is a face of $\sigma_i$. Then by \cite[\S3.1 p. 53]{Fu1993Introduction-to-to},
	it follows that $V(\eta)$ and $X_{\sigma_i}$ do intersect. In particular, 
	$V(\eta) \not\subset X_{\aff} \setminus X$. 
	Since $X_{\aff} \setminus X$ is a closed $T$-invariant subset,
	it is the union of some $V(\varepsilon)$ for some faces $\varepsilon$ of $\sigma$.
	This implies then the claim.
	
	Statement~\eqref{lem.toric_qusi-affine3} now follows from the claim,
	since the minimal faces $\tau \subset \sigma$ with $\tau \not\in \Sigma$
	correspond to the maximal $V(\tau)$ in $X_{\aff} \setminus X$.
	
	\eqref{lem.toric_qusi-affine4}: Since $X$ is quasi-affine it follows from~\eqref{lem.toric_qusi-affine2}
	that each $\sigma_i$ is a face of $\sigma$. Since $\sigma$ is the convex hull of the $\sigma_i$,
	we get thus that the extremal rays of $\sigma$ are the same as the extremal rays of all the $\sigma_i$.
	Hence the extremal rays of $\sigma$ are the same as the cones of dimension one in $\Sigma$. In particular,
	each face $\tau$ of $\sigma$ with $\tau \not\in \Sigma$ has dimension at least $2$.
\end{proof}

For the description of the homogeneous $\GG_a$-actions, let us setup the following notation. Let $\sigma \subset N_{\RR}$ be a strongly convex rational polyhedral cone. 
If $\rho \subset \sigma$ is an extremal ray
and $\tau \subset \sigma$ a face, we denote
\[
	\tau_\rho := \Conv(\textrm{extremal rays in $\tau$ except $\rho$}) \subset N_{\RR} \, .
\]
In particular, if $\rho$ is not an extremal ray of $\tau$, then $\tau_\rho = \tau$. Let us mention the following 
easy observations of this construction for future use:

\begin{lemma}
	\label{lem.tau_rho_is_a_face}
	Let $\sigma \subset N_{\RR}$ be a strongly convex rational polyhedral cone, $\tau \subset \sigma$
	a face and $\rho \subset \sigma$ an extremal ray. Then
	\begin{enumerate}[leftmargin=*]
		\item \label{lem.tau_rho_is_a_face1} $\tau_\rho$ is a face of $\sigma_\rho$;
		\item \label{lem.tau_rho_is_a_face2} If $\dim \tau_\rho < \dim \tau $, then $\tau_\rho$ is a face of $\tau$.
	\end{enumerate}
\end{lemma}

\begin{proof}
	\eqref{lem.tau_rho_is_a_face1}:
	By definition, there is $u \in \sigma^\vee$ with $\tau = \sigma \cap u^\bot$. Hence 
	$\tau_\rho \subset \sigma_\rho \cap u^\bot \subset \tau$. Since $u \in (\sigma_\rho)^\vee$,
	$\sigma_\rho \cap u^\bot$ is a face of $\sigma_\rho$.
	If $\rho \not\subset \tau$, then $\tau_\rho = \tau$
	and thus $\tau_\rho = \sigma_\rho \cap u^\bot$ is a face of $\sigma_\rho$. If $\rho \subset \tau$,
	then $\sigma_\rho \cap u^\bot$ is the convex cone generated by the extremal rays in $\tau$, except $\rho$, i.e.
	$\tau_\rho = \sigma_\rho \cap u^\bot$. Thus $\tau_\rho$ is a face of $\sigma_\rho$.
	
	\eqref{lem.tau_rho_is_a_face2}: As 
	$\dim \tau_\rho < \dim \tau$, we get $\rho \subset \tau$ and
	\[
		\label{eq.Span_tau}
		\tag{$\boxbox$}
		\Span_{\RR}(\tau) = \RR \rho \oplus \Span_{\RR}(\tau_\rho) \, .
	\]
	Hence, there is $u \in M$ such that $\Span_{\RR}(\tau_\rho) = u^\bot \cap \Span_{\RR}(\tau)$.
	After possibly replacing $u$ by $-u$, we may assume $\sprod{u}{v_\rho} \geq 0$ where
	$v_\rho \in \rho$ denotes the unique primitive generator. 
	As $\tau_\rho \subset u^\bot$, we get now $u \in \tau^\vee$. Moreover, 
	\[
		u^\bot \cap \tau = \Span_{\RR}(\tau_\rho) \cap \tau = \tau_\rho
	\]
	where the second equality follows from~\eqref{eq.Span_tau} as one may write each element in $\tau$
	as $\lambda v_{\rho} + \mu w$ for $w \in \tau_\rho$ and $\lambda, \mu \geq 0$.
	Thus $\tau_\rho$ is a face of $\tau$.
\end{proof}

For each extremal ray $\rho$ in a strongly convex rational polyhedral cone $\sigma$, let 
\[
	S_{\rho} := \set{m \in (\sigma_\rho)_M^\vee}{\sprod{m}{v_\rho} = -1}
\]
where $v_\rho \in \rho$ denotes the unique primitive generator. 

\begin{remark}[{see also \cite[Remark~2.5]{Li2010Affine-Bbb-T-varie}}]
	\label{rem.S_rho_non-empty}
	The set $S_\rho$ is non-empty. Indeed, apply Proposition~\ref{prop.Weights_facet}\eqref{prop.Weights_facet1}
	to the convex polyhedral cone $C = \sigma_\rho^\vee$ and the hyperplanes $H = \rho^\bot$, 
	$H' = \set{u \in M_\RR}{\sprod{u}{v_\rho} = -1}$ 
	inside $V = M_\RR$.  
\end{remark}

Now, we come to the promised description
of the homogeneous $\GG_a$-actions on toric varieties due to \name{Liendo}:


\begin{proposition}[{\cite[Lemma~2.6, Theorem~2.7]{Li2010Affine-Bbb-T-varie}}]
	\label{prop.G_a-actions_on_affine_toric_var}
	Let $\sigma \subset N_\RR$ be a strongly convex rational polyhedral cone. 
	Then for any extremal ray $\rho$ in $\sigma$ and any 
	$e \in S_\rho$, the linear map
	\[
		\partial_{\rho, e} \colon \kk[\sigma^\vee_M] \to \kk[\sigma^\vee_M]
		\, , \quad
		\chi^m \mapsto \langle m, v_\rho \rangle \chi^{e+m}
	\]
	is a homogeneous locally nilpotent derivation of degree $e$,
	and every homogeneous locally nilpotent derivation of $\kk[\sigma^\vee_M]$ 
	is a constant multiple of some $\partial_{\rho, e}$.
\end{proposition}

\begin{remark}
	The weight of the homogeneous $\GG_a$-action induced by $\partial_{\rho, e}$ is $e \in M$.
	The kernel of the locally nilpotent derivation $\partial_{\rho, e}$ is $\kk[\sigma_M^\vee \cap \rho^\bot]$.
\end{remark}

The following lemma is the key for the description of the homogeneous $\GG_a$-actions on a 
quasi-affine toric variety.

\begin{lemma}
	\label{Lem:Orbit_closure_invariant}
	Let $\sigma \subset N_{\RR}$ be a strongly convex rational polyhedral cone, $\tau \subset \sigma$ a face, 
	$\rho \in \sigma$ an extremal ray and $e \in S_\rho$. Then the
	$\GG_a$-action on $X_\sigma$ corresponding to the locally nilpotent derivation $\partial_{\rho, e}$ 
	leaves $V(\tau)$ invariant if and only if 
	\[
		\rho \not\subset \tau \quad \textrm{or} \quad e \not\in \tau_{\rho}^\bot \, .
	\]
\end{lemma}

\begin{proof}
	As in the proof of~Lemma~\ref{lem.toric_qusi-affine}~\eqref{lem.toric_qusi-affine3},  
	the embedding
	$\iota \colon V(\tau) \subset X_\sigma$ corresponds to the surjective $\kk$-algebra homomorphism
	\[
		\iota^\ast \colon \kk[\sigma^{\vee}_M] \to \kk[\sigma^{\vee}_M \cap \tau^{\bot}] \, , 
		\quad 
		\chi^m \mapsto \left\{\begin{array}{rl}
		\chi^m \, , & \textrm{if $m \in \tau^{\bot}$} \\
		0 \, , & \textrm{if $m \in \sigma_M^\vee \setminus \tau^\bot$}
		\end{array}\right. \, .
	\] 
	Thus the $\GG_a$-action on $X_\sigma$ corresponding to $\partial_{\rho, e}$ 
	preserves $V(\tau)$ if and only
	if
	\[
		\partial_{\rho, e}(\ker \iota^\ast) \subset \ker \iota^\ast  \, ,
	\]
	see \cite[Sect. 1.5]{Fr2017Algebraic-theory-o}.
	Since $\sprod{m}{v_\rho} = 0$ for all $m \in \rho^\bot$  
	and since $e + m \in \sigma^\vee_M$ for all $m \in \sigma^\vee_M \setminus \rho^\bot$,
	this last condition is equivalent to 
	\[
		\label{Eq:Condition_Invariant}
		\tag{$\odot$}
		m \in \sigma_M^\vee \setminus (\tau^\bot \cup \rho^\bot) \quad \implies \quad
		e+m \not\in \tau^\bot \, .
	\]
	We distinguish now two cases:
	\begin{enumerate}[leftmargin=*]
		\item Assume $\rho \not\subset \tau$. Then $\tau \subset \sigma_\rho$. In particular, we
		get $\sprod{e}{v} \geq 0$ for all $v \in \tau$. 
		Let $m \in \sigma_M^\vee \setminus \tau^\bot$. Then
		we get $\sprod{m}{v} > 0$ for some $v \in \tau$ and hence
		\[
			\sprod{e+m}{v} > 0 \quad \textrm{for some $v \in \tau$}
		\]
		which in turn implies $e + m \not\in \tau^\bot$. Thus
		\eqref{Eq:Condition_Invariant} is satisfied.
		\item Assume $\rho \subset \tau$. In  particular, we have $\tau^\bot \subset \rho^\bot$.
		We distinguish two cases:
		\begin{itemize}[leftmargin=*]
			\item $e \not\in \tau_{\rho}^\bot$: Then there exists an
			extremal ray $\rho' \subset \tau$ with $\rho' \neq \rho$ such that $e \not\in (\rho')^\bot$ and 
			the unique primitive generator
			$v_{\rho'} \in \rho'$ satisfies
			\[
			\sprod{e+m}{v_{\rho'}} = \underbrace{\sprod{e}{v_{\rho'}}}_{>0} + \underbrace{\sprod{m}{v_{\rho'}}}_{\geq 0}
			> 0  \quad \textrm{for all $m \in \sigma_M^\vee$} \, .
			\]
			In particular $e+m \not\in \tau^\bot$
			for all $m \in \sigma_M^\vee$ and thus~\eqref{Eq:Condition_Invariant} is satisfied.
			
			\item $e \in \tau_\rho^\bot$:
			Now, we want to apply Proposition~\ref{prop.help_lemma_weights}.
			For this we fix the lattice $\Lambda = M \cap \tau_{\rho}^\bot$ inside 
			$V = \tau_{\rho}^\bot$. 
			Since $\tau_\rho$ is a face of $\sigma_\rho$ (see Lemma~\ref{lem.tau_rho_is_a_face}\eqref{lem.tau_rho_is_a_face1}), 
			$C =(\sigma_\rho)^\vee \cap \tau_{\rho}^\bot$ is a rational convex polyhedral cone
			in $M_\RR$ and thus also in $V$. 
			Moreover, we set $H_0 = \rho^\bot \cap V = \tau^\bot$
			and $H_{\pm1} = \set{u \in V}{\sprod{u}{v_{\rho}} = \pm 1}$.
			Since $e \in (S_\rho \cap \tau_\rho^\bot) \setminus H_0$, $H_0$ is a hyperplane in $V$ and
			$C \cap H_{-1} \cap \Lambda = S_{\rho} \cap \tau_{\rho}^\bot \neq \varnothing$.
			Since $H_0 \subsetneq V$ we get thus $\dim \tau_\rho < \dim \tau$.
			Now, by Lemma~\ref{lem.tau_rho_is_a_face}\eqref{lem.tau_rho_is_a_face2}, 
			$\tau_\rho$ is a face of $\tau$ and therefore
			$\sigma^\vee \cap \tau_\rho^\bot \supsetneq \sigma^\vee \cap \tau^\bot$
			by the order-reversing bijection between faces of $\sigma$ and 
			$\sigma^\vee$. Hence, 
			there is $u \in (\sigma^\vee \cap \tau_\rho^\bot) \setminus \tau^\bot$ and in particular 
			$u \in C \setminus H_0$. As $\sprod{u}{v_{\rho}} > 0$, after scaling $u$ 
			with a real number $>0$, we may assume
			$u \in C \cap H_1$ and hence $C \cap H_1 \neq \varnothing$.
			Now, as $C \cap H_0 = \sigma^\vee \cap \tau^{\bot}$ is rational
			in $M_{\RR}$ and thus also in $V$, 
			we may apply Proposition~\ref{prop.help_lemma_weights} and get some 
			\[
				m_1 \in (\sigma_\rho)^\vee \cap \tau_{\rho}^\bot \cap \set{m \in M}{\sprod{m}{v_{\rho}}=1} \, .
			\]
			Hence, $m_1 \in \sigma_M^\vee \setminus \rho^\bot$. 
			Since $e, m_1 \in \tau_{\rho}^\bot$, we get
			$e + m_1 \in \tau_{\rho}^\bot$. Since
			\[
				\sprod{e + m_1}{v_{\rho}} = \sprod{e}{v_{\rho}} + \sprod{m_1}{v_{\rho}} = -1 + 1 = 0 \, ,
			\]
			we get thus $e + m_1 \in \tau^\bot$. This implies 
			that~\eqref{Eq:Condition_Invariant} is not satisfied.
		\end{itemize}
	\end{enumerate}
\end{proof}

We can use this lemma to give a full description of all homogeneous $\GG_a$-actions on a quasi-affine toric variety $X = X_\Sigma$ where $\Sigma$ is a fan in $N_{\RR}$. 
Recall that $X_{\aff} = X_\sigma$, where $\sigma$ is the
cone in $N_{\RR}$ generated by all maximal cones in $\Sigma$, see 
Lemma~\ref{lem.toric_qusi-affine}. Moreover, $X_{\aff} \setminus X$ is the union of the
sets of the form $V(\tau)$ where $\tau \subset \sigma$ runs through the minimal faces with 
the property that $\tau \not\in \Sigma$ (again by Lemma~\ref{lem.toric_qusi-affine}).
In the next corollaries, we use this notation freely.

\begin{corollary}
	\label{Cor:Classification_LND_Quasi-affine}
	Let $X = X_{\Sigma}$ be a quasi-affine toric variety, let $X_{\aff} = X_\sigma$ and let $\tau_1, \ldots, \tau_s \subset \sigma$
	be the minimal faces of $\sigma$ with $\tau_i \not\in \Sigma$ for all $i=1, \ldots, s$. Then, the homogeneous
	$\GG_a$-actions on $X$ are the restricted homogeneous $\GG_a$-actions on $X_{\aff}$ that 
	are induced by the constant multiples of $\partial_{\rho, e} \in \LND_{\kk}(\OO(X))$
	such that
	for all $i = 1, \ldots, s$ we have
	\[
		\label{Eq:Classification_LND_Quasi-affine}
		\tag{$\circledast$}
		\rho \not\subset \tau_i \quad \textrm{or} \quad e \not\in (\tau_i)_{\rho}^\bot \, .
	\]
\end{corollary}

\begin{proof}
	Assume that $\partial_{\rho, e}$ is a locally nilpotent derivation of $\OO(X)$
	such that~\eqref{Eq:Classification_LND_Quasi-affine} is satisfied for all $i =1, \ldots, s$. 
	Then by Lemma~\ref{Lem:Orbit_closure_invariant}, the sets $V(\tau_1), \ldots, V(\tau_s) \subset X_{\aff}$
	are left invariant by the homogeneous $\GG_a$-action 
	$\varepsilon_{\rho, e} \colon \GG_a \times X_{\aff} \to X_{\aff}$ 
	which is induced by $\partial_{\rho, e}$. 
	In particular, $X = X_{\aff}Ê\setminus (V(\tau_1) \cup \ldots V(\tau_s))$ (see Lemma~\ref{lem.toric_qusi-affine}) is left 
	invariant by $\varepsilon_{\rho, e}$. 
	
	On the other hand, let $\varepsilon \colon \GG_a \times X \to X$ be a homogeneous $\GG_a$-action on $X$.
	By Lemma~\ref{lem:Quasi-affine_group_action} and Proposition~\ref{prop.G_a-actions_on_affine_toric_var} 
	this $\GG_a$-action extends to a homogeneous $\GG_a$-action
	$\varepsilon_{\rho, e} \colon \GG_a \times X_{\aff} \to X_{\aff}$ which is induced
	by some locally nilpotent derivation $\lambda \cdot \partial_{\rho, e} \in \LND_{\kk}(\OO(X))$ for some
	constant $\lambda \in \kk$, some
	extremal ray $\rho$ in $\sigma$ and some $e \in S_{\rho}$. Since $\varepsilon_{\rho, e}$ extends
	$\varepsilon$, the subset $V(\tau_1) \cup \ldots \cup V(\tau_s) = X_{\aff} \setminus X$ is left
	invariant by $\varepsilon_{\rho, e}$. Since the $V(\tau_1), \ldots, V(\tau_s)$ are the 
	irreducible components of $X_{\aff} \setminus X$ and since 
	$\GG_a$ is an irreducible algebraic group, it follows that $\varepsilon_{\rho, e}$ preserves each $V(\tau_i)$.
	By Lemma~\ref{Lem:Orbit_closure_invariant} we get that for each $i=1, \ldots, s$ 
	the condition~\eqref{Eq:Classification_LND_Quasi-affine} is satisfied.
\end{proof}

For the next consequences of Corollary~\ref{Cor:Classification_LND_Quasi-affine} we recall the 
following notation from Sect.~\ref{sec.Cones_and_asymptotic_cones}: For a
subset $E \subset M_{\RR}$ we denote by $\inter(E)$ the topological interior of $E$ 
inside the linear span of $E$. In these consequences we give a closer
description of the weights in $M$ arising from homogeneous $\GG_a$-actions on quasi-affine toric varieties
and compute the asymptotic cone of these weights.

\begin{corollary}
	\label{cor.Weights_rho}
	Let $X = X_{\Sigma}$ be a quasi-affine toric variety, let $X_{\aff}=X_{\sigma}$, let 
	$\rho \subset \sigma$ be an extremal ray and let $D_{\rho}(X)$
	be the set of weights $e \in S_\rho$ such that the locally nilpotent derivation 
	$\partial_{\rho, e}$ of $\OO(X)$ induces a homogeneous $\GG_a$-action on $X$. Then
	\[ 
		S_{\rho} \cap \inter(\sigma_{\rho}^\vee) \subset
		D_{\rho}(X) \subset S_\rho \, .
	\]
\end{corollary}

\begin{proof}
	Let $e \in S_{\rho} \subset M$ such that $e$ is contained in 
	$\inter(\sigma_{\rho}^\vee) \subset M_{\RR}$.
	Let $\tau_1, \ldots, \tau_s$ be the minimal faces of $\sigma$ which are not contained in $\Sigma$.
	According to Corollary~\ref{Cor:Classification_LND_Quasi-affine} it is enough to show that for each $\tau_i$
	with $\rho \subset \tau_i$ we have $e \not\in (\tau_i)_{\rho}^\bot$. By 
	Lemma~\ref{lem.toric_qusi-affine}~\eqref{lem.toric_qusi-affine4} we get that $\dim \tau_i \geq 2$
	for every $i$. Hence, $\dim (\tau_i)_\rho \geq 1$
	and thus $(\tau_i)_\rho^\bot \cap \sigma_\rho^\vee$ is a proper face of $\sigma_\rho^\vee$. As
	$e \in \inter(\sigma_\rho^\vee)$, we get 
	$e \not\in (\tau_i)_\rho^\bot \cap \sigma_\rho^\vee$ and thus $e \not\in (\tau_i)_\rho^\bot$.
\end{proof}

\begin{corollary}
	\label{cor.Weights}
	Let $X = X_\Sigma$ be a quasi-affine toric variety. Let $X_{\aff} = X_{\sigma}$ 
	and let $D(X)$ be the set of weights $e \in M$ such that there
	is a non-trivial homogeneous $\GG_a$-action on $X$ of weight $e$. Then the asymptotic cone of 
	$D(X) \subset M_{\RR}$ satisfies
	\[
		D(X)_{\infty} = \sigma^\vee \setminus \inter(\sigma^\vee) \, .
	\]
\end{corollary}

\begin{proof}
	We have
	\[
		\sigma^\vee \setminus \inter(\sigma^\vee) = \bigcup_{\substack{
			\textrm{$\rho$ is an extr.} \\
			\textrm{ray of $\sigma$}
		}} \sigma^\vee \cap \rho^\bot \, .
	\]
	Since $D(X)$ is the union of the $D_{\rho}(X)$ for the extremal rays $\rho \subset \sigma$
	(with the definition of $D_\rho(X)$ from Corollary~\ref{cor.Weights_rho}), 
	we get by Lemma~\ref{lem.Properties_asymptotic_cone} that
	\[
		D(X)_{\infty} = \bigcup_{\substack{
				\textrm{$\rho$ is an extr.} \\
				\textrm{ray of $\sigma$}
		}} D_{\rho}(X)_{\infty} \, .
	\]
	Now, we want to apply Proposition~\ref{prop.Weights_facet} for an extremal ray $\rho$ in $\sigma$. 
	For this we fix the lattice $\Lambda = M$ inside $V = M_{\RR}$ and consider
	the convex polyhedral cone $C =\sigma_{\rho}^\vee$ inside $V$
	and the hyperplane $H = \rho^\bot \subset V$. Note that
	$C \cap H = \sigma^\vee \cap \rho^\bot$ is a rational convex polyhedral cone in $V$ of dimension
	$\dim H$ and that $H \cap \Lambda$ has rank $\dim H$. 
	Moreover there exists $m_{-1} \in M \setminus H$ such that 
	$H' = \set{u \in M_{\RR}}{\sprod{u}{v_{\rho}} =-1} = m_{-1} + H$ where $v_\rho \in \rho$
	denote the unique primitive generator. Since $\rho$ is an extremal ray
	of $\sigma$, it follows that $\sigma_{\rho} \subsetneq \sigma$ and thus 
	$\sigma_{\rho}^\vee \supsetneq \sigma^\vee = 
	\sigma_{\rho}^\vee \cap \set{u \in M_{\RR}}{\sprod{u}{v_{\rho}}\geq 0}$. This implies that
	there is $u \in C$ with $\sprod{u}{v_{\rho}}<0$. Since $C$
	is a cone, we get that $C \cap H'$ is non-empty. Now, Proposition~\ref{prop.Weights_facet} applied to $\Lambda, C, H, H' \subset V$ 
	implies that
	\[
		\label{eq.1}
		\tag{$d$}
		\sigma^\vee \cap \rho^\bot = \sigma_{\rho}^\vee \cap \rho^\bot = 
		(S_{\rho} \cap \inter(\sigma_{\rho}^\vee))_{\infty} \, .
	\]
	By Corollary~\ref{cor.Weights_rho} and Lemma~\ref{lem.Properties_asymptotic_cone} we get
	\[
		\label{eq.2}
		\tag{$e$}
		(S_{\rho} \cap \inter(\sigma_{\rho}^\vee))_{\infty} \subset D_{\rho}(X)_{\infty}
		\subset (S_{\rho})_{\infty} 
		\subset (\sigma_{\rho}^\vee \cap (m_{-1} + \rho^\bot))_{\infty}
		\subset \sigma_{\rho}^\vee \cap \rho^\bot
		\, .
	\]
	Combining~\eqref{eq.1} and~\eqref{eq.2} yields
	$\sigma^\vee \cap \rho^\bot = D_{\rho}(X)_{\infty}$ which implies the result.
\end{proof}

In the sequel we also need that $D(X)$ generates $M$ as a group.  

\begin{proposition}
	\label{prop.DX_generates_M_as_a_group}
	Let $X$ be a quasi-affine toric variety and let $D(X)$ be the set of all weights $e \in M$
	such that there is a non-trivial 
	homogeneous $\GG_a$-action on $X$ of weight $e$. If 
	$X \not \simeq T$, then $D(X)$ generates $M$ as a group.
\end{proposition}

Since $X \not\simeq T$, the cone $\sigma$ has an extremal ray $\rho$.
Proposition~\ref{prop.DX_generates_M_as_a_group} follows thus from the next lemma, since 
$S_\rho \cap \inter(\sigma_{\rho}^\vee) \subset D(X)$ (see Corollary~\ref{cor.Weights_rho}).

\begin{lemma}
	Let $\sigma \subset N_{\RR}$ be a strongly convex rational polyhedral cone. Then for every
	extremal ray $\rho \subset \sigma$, the set $S_{\rho} \cap \inter(\sigma_{\rho}^\vee)$ generates $M$ as a group. 
\end{lemma}

\begin{proof}
	Denote by $v_\rho \in \rho$ the unique primitive generator.
	By Remark~\ref{rem.S_rho_non-empty}, $S_\rho$ is non-empty.
	Thus by Proposition~\ref{prop.Weights_facet}\eqref{prop.Weights_facet1} 
	applied to the convex polyhedral cone $C = \sigma_{\rho}^\vee$
	and the hypersurfaces $H = \rho^\bot$, $H' = \set{u \in V}{\sprod{u}{v_{\rho}}=-1}$ in $V = M_{\RR}$
	we get $S_{\rho}  \cap \inter(\sigma_{\rho}^\vee) \neq \varnothing$.  
	Let $A = S_{\rho}  \cap \inter(\sigma_{\rho}^\vee)$ and choose $a \in A$.
	By definition of $S_{\rho}$,
	\[
		a + (\sigma_M^\vee \cap \rho^\bot) \subset A \, .
	\]
	Since $v_{\rho} \in N$ is primitive, 
	we can choose coordinates $N = \ZZ^n$ (where $n = \rank N$) such 
	that $v_{\rho} = (1, 0, \ldots, 0)$. We then identify $M = \Hom_{\ZZ}(N, \ZZ)$
	with $\ZZ^n$ by choosing the dual basis of $N = \ZZ^n$.
	Since $\sigma^\vee \cap \rho^\bot$ is a convex rational polyhedral cone of dimension 
	$\dim \rho^\bot$ in $\rho^\bot$, there is $m \in \sigma_M^\vee \cap \rho^\bot$ such that
	the closed ball of radius $1$ and centre $m$ in $\rho^\bot$ is contained in 
	$\sigma^\vee \cap \rho^\bot$. In particular, $m + e_i \in \sigma_M^\vee \cap \rho^\bot$
	for $i =2, \ldots, n$, where $e_i = (0, \ldots, 0, 1, 0, \ldots, 0)$ and $1$ is at position $i$. 
	In particular,
	\[
		e_i = (a + m + e_i) - (a+m) \in \Span_{\ZZ}(A) \quad \textrm{for $i=2, \ldots, n$} \, .
	\]
	Since $v_\rho = (1, 0, \ldots, 0)$ and $\sprod{a}{v_{\rho}} = -1$, it follows that
	$a = (-1, a_2, \ldots, a_n)$ for certain $a_2, \ldots, a_n \in \ZZ$. 
	In particular, $(1, 0, \ldots, 0) = -a + \sum_{i=2}^n a_i e_i \in \Span_{\ZZ}(A)$.
	Thus, $\Span_{\ZZ}(A) = M$.
\end{proof}


\section{The automorphism group determines sphericity}
\label{sec.determination_of_sphericity}


Our first goal in this section is to give a criterion, when a solvable algebraic group $B$
acts with an open orbit on a quasi-affine $B$-variety.
For this, we introduce the notion of \emph{generalized root
subgroups}:

\begin{definition}
	Let $H$ be an algebraic group and let $X$ be a $H$-variety. 
	We call an algebraic subgroup $U_0 \subseteq \Aut(X)$ of dimension $m$ 
	which is isomorphic to $(\GG_a)^m$ a
	\emph{generalized root subgroup (with respect to $H$)} 
	if there exists a character $\lambda \in \frak{X}(H)$, called the \emph{weight of $U_0$} such that
	\[
		h \circ \varepsilon(t) \circ h^{-1} = \varepsilon(\lambda(h) \cdot t) \quad 
		\textrm{for all $h \in H$ and all $t \in (\GG_a)^m$}
	\]
	where $\varepsilon \colon (\GG_a)^m \simeq U_0$ is a fixed isomorphism.
\end{definition}

Using that a group automorphism of $(\GG_a)^m$ is $\kk$-linear, we
see that the weight of a generalized root subgroup $U_0$ does not depend on the choice 
of an isomorphism $\varepsilon \colon (\GG_a)^m \simeq U_0$.

\begin{remark}
	\label{rem.equivalent_def_gen_root_subgroup}
	Let $H$ be an algebraic group and let $X$ be an $H$-variety. 
	Using again that algebraic 
	group automorphisms of $(\GG_a)^m$ are $\kk$-linear, one can see the following:
	An algebraic subgroup $U_0 \subset \Aut(X)$
	which is isomorphic to $(\GG_a)^m$ for some $m \geq 1$
	is a generalized root subgroup with respect to $H$
	if and only if each one-dimensional closed subgroup of $U_0$ is a root subgroup of $\Aut(X)$ 
	with respect to $H$.
	In particular, root subgroups are generalized root subgroups of dimension one.
\end{remark}


\begin{proposition}
	\label{prop:characterization_of_aff_spherical_var}
	Let $B$ be a connected solvable algebraic group that contains unipotent elements and
	let $X$ be an irreducible quasi-affine variety with a faithfull $B$-action. 
	Then, the following statements are equivalent:
	\begin{enumerate}[leftmargin=*]
		\item \label{item:Open_B-orbit} The variety $X$ has an open $B$-orbit;
		\item \label{item:Boundedness_of_homogeneous_LNDs}
		There is a constant $C$ such that $\dim \Ve(X)_{\lambda} \leq C$
		for all weights $\lambda \in \frak{X}(B)$; 
		\item \label{item:Commutative_unipot_subgroups}
		There exists a constant $C$ such that 
		$\dim U_0 \leq C$ for each $U_0 \subseteq \Aut(X)$ which is a
		generalized root subgroup with respect to $B$.
	\end{enumerate}
\end{proposition}

\begin{proof}
	\eqref{item:Open_B-orbit}$\implies$\eqref{item:Boundedness_of_homogeneous_LNDs}:
	By Lemma~\ref{lem:Normalized_Derivations} we get 
	$\dim \Ve(X)_{\lambda} \leq \dim T_{x_0} X$ where $x_0 \in X$ is a fixed element
	of the open $B$-orbit. 
	
	\eqref{item:Boundedness_of_homogeneous_LNDs}
	$\implies$\eqref{item:Commutative_unipot_subgroups}: Let $U_0 \subseteq \Aut(X)$
	be a generalized root subgroup of weight $\lambda \in \frak{X}(B)$. 
	By Lemma~\ref{lem.faithfull_action}, the $\kk$-linear map 
	$\Lie(U_0) \to \Ve(X)$, $A \mapsto \xi_A$ is injective. 
	Now, take $A \in \Lie(U_0)$ which is non-zero. Then there is a one-parameter
	unipotent subgroup $U_{0, A} \subset U_0$ such that $\Lie(U_{0, A})$ is generated by $A$.
	By definition, $U_{0, A}$ is a root subgroup with respect to $B$ of weight $\lambda$.
	By Lemma~\ref{lem.H-homogeneous_Ga-action}, it follows that $\xi_A$ lies in $\Ve(X)_{\lambda}$.
	Thus the whole image of $\Lie(U_0) \to \Ve(X)$ lies in $\Ve(X)_\lambda$ and we get
	$\dim U_0 \leq \dim \Ve(X)_\lambda$.
	
	\eqref{item:Commutative_unipot_subgroups}$\implies$\eqref{item:Open_B-orbit}:
	Assume that $X$ admits no open $B$-orbit.
	This implies by Rosenlicht's Theorm~\cite[Theorem~2]{Ro1956Some-basic-theorem} 
	that there is a $B$-invariant non-constant
	rational map $f \colon X \dasharrow \kk$.	
	By Proposition~\ref{prop:H-semi-invariants}, 
	there exist $B$-semi-invariant regular functions $f_1, f_2 \colon X \to \kk$ 
	such that $f = f_1 / f_2$ and since $f$ is $B$-invariant, the weights of 
	$f_1$ and $f_2$ under $B$ are the same, say $\lambda_0 \in \frak{X}(B)$. 
	
	Moreover, there exists no non-zero homogeneous
	polynomial $p$ in two variables with  $p(f_1, f_2) = 0$. Indeed, otherwise
	there exist $m > 0$ and a non-zero tuple $(a_0, \ldots, a_m) \in \kk^{m+1}$ such that
	$\sum_{i=0}^m a_i (f_1)^i (f_2)^{m-i} = 0$ and hence 
	$\sum_{i=0}^m a_i f^i = 0$. Since $f$ is non-constant, we get
	a contradiction, as $\kk$ is algebraically closed.
	
	Since $B$ contains unipotent elements, the centre of the unipotent radical in $B$ is non-trivial.
	Let $\rho \colon \GG_a \times X \to X$ be the $\GG_a$-action
	that corresponds to a one-dimensional subgroup of this centre Hence 
	$\rho$ is $B$-homogeneous for some weight $\lambda_1 \in \frak{X}(B)$.
	Thus for any $m \geq 0$, we get a faithful $(\GG_a)^{m+1}$-action on $X$ given by
	\[
		\GG_a^{m+1} \times X \to X \, , \quad 
		((t_0, \ldots, t_m), x) \mapsto 
		\rho\left(\sum_{i=0}^m t_i (f_1^i \cdot f_2^{m-i})(x), x \right) \, ,
	\]
	since $\sum_{i=0}^m t_i f_1^i \cdot f_2^{m-i} \neq 0$ for all non-zero
	$(t_0, \ldots, t_m)$.
	The corresponding subgroup $U_0$ in $\Aut(X)$ is then 
	a generalized root subgroup of dimension $m+1$
	with respect to $B$ of weight $\lambda_1 + m \lambda_0 \in \frak{X}(B)$.
	As $m$ was arbitrary, \eqref{item:Commutative_unipot_subgroups} is not satisfied.
\end{proof}

\begin{example}
	If the connected solvable algebraic group $B$ does not contain unipotent elements, then 
	Proposition~\ref{prop:characterization_of_aff_spherical_var} is in general false: Let 
	$B = \GG_m$ act on $X = \AA^2$ via $t \cdot (x, y) = (x, ty)$ and consider for
	each $m \geq 0$ the faithful $(\GG_a)^{m+1}$-action:
	\[
		\GG_a^{m+1} \times X \to X \, , \quad 
		((t_0, \ldots, t_m), x) \mapsto 
		(x, y + t_0 + t_1 x + \ldots + t_m x^m) \, .
	\]
	The corresponding subgroup in $\Aut(X)$ is then a generalized 
	root subgroup with respect to $B$ of weight $\id \colon \GG_m \to \GG_m$, $s \mapsto s$.
\end{example}

\begin{lemma}
	\label{lem.open_orbit_of_solvable_group}
 	Let $T$ be an algebraic torus and let $X$ be a quasi-affine 
	$T$-toric variety such that $X \not\simeq T$.
 	Then there exists a non-trivial $T$-homogeneous $\GG_a$-action on $X$ and an
 	subtorus $T' \subset T$ of codimension one such that the induced
 	$\GG_a \rtimes T'$-action on $X$ has an open orbit.
\end{lemma}

\begin{proof}
	Since $X \not\simeq T$, 
	there is a non-trivial $T$-homogeneous $\GG_a$-action on $X$ 
	by Proposition~\ref{prop.DX_generates_M_as_a_group}. Denote by 
	$V \subset \Aut(X)$ the corresponding
	root subgroup.

	Let $x_0 \in X$ such that $T x_0 \subset X$ is an open in $X$ and let 
	$S$ be the connected component of the stabilizer in $V \cdot T$ of $x_0$. 
	As $\dim V \cdot T = \dim X + 1$, we get $\dim S = 1$.
	If $S$ would be contained in $V$, 
	then $S = V$ and thus $v x_0 = x_0$ for all  $v \in V$. From this we would get for all $t \in T$, $v \in V$ that
	\[
		(t v t^{-1}) \cdot (t x_0) = t x_0 
	\]
	and hence $V$ would fix each element of the open orbit $T x_0$, contradiction. 
	Hence $S \not\subset V$, which implies that there is a codimension one
	subgroup $T' \subset T$ with $S \not\subset V \cdot T'$. This implies that
	$(V \cdot T') \cap S$ is finite
	and thus $V \cdot T' x_0$ is dense in $X$. As orbits are locally closed, we get that
	$V \cdot T' x_0$ is open in $X$.
\end{proof}

For the sake of completeness let us recall the following well-known fact from the theory of
algebraic groups:

\begin{lemma}
	\label{lem.B_contains_unipotent_elements}
	Let $G$ be a connected reductive algebraic group and let $B \subset G$ be a Borel subgroup.
	If $G$ is not a torus, then $B$ contains non-trivial unipotent elements.
\end{lemma}

\begin{proof}
	Denote by $T \subset B$ a maximal torus.
	Towards a contradiction, assume that $B$ contains no non-trivial
	unipotent element. Then
	$B = T$ and by 
	the Bruhat decomposition 
	(see \cite[Theorem in \S28.3]{Hu1975Linear-algebraic-g}), it follows
	that $G$ is equal to the normalizer $N_G(T)$ of $T$ in $G$. 
	Since the Weyl group $N_G(T)/T$ is finite, $T$ has finite index in $G$ and by the
	connectivity of $G$ we get $G = T$, contradiction.
\end{proof}

Now, we prove that one can recognize the spericity of an irreducible
quasi-affine normal $G$-variety from its automorphism group.

\begin{proposition}
	\label{prop.iso_and_sphericity}
	Let $G$ be a connected reductive algebraic group and
	let $X$, $Y$ be irreducible quasi-affine normal varieties.
	Assume that there is a group isomorphism 
	$\theta \colon \Aut(X) \to \Aut(Y)$ that preserves algebraic subgroups. 
	If $X$ is non-isomorphic to a torus and $G$-spherical,
	then $Y$ is $G$-spherical for the induced $G$-action via $\theta$. 
\end{proposition}

\begin{proof}
	We denote by $B \subseteq G$
	a Borel subgroup and by $T \subseteq G$ a maximal torus. We distinguish two cases:
	
		$G \neq T$: By Lemma~\ref{lem.B_contains_unipotent_elements} the Borel subgroup $B$
		contains unipotent elements and thus we may apply
		Proposition~\ref{prop:characterization_of_aff_spherical_var} in order to get
		a bound on  the dimension of every generalized root subgroup with respect to $B$ of $\Aut(X)$.
		Since the generalized root subgroups of $\Aut(X)$ (with respect to $B$)
		correspond bijectively to the generalized root subgroups of $\Aut(Y)$ 
		(with respect to $\theta(B)$) via $\theta$ (see Remark~\ref{rem.equivalent_def_gen_root_subgroup}
		and Lemma~\ref{lem.auto_perserves_weights}), it follows by 
		Proposition~\ref{prop:characterization_of_aff_spherical_var} that
		$Y$ is $\theta(G)$-spherical.
		
		$G = T$: In this case $X$ is $T$-toric. Since $X$ is not isomorphic to a torus, we may apply
		Lemma~\ref{lem.open_orbit_of_solvable_group} in order to get a codimension one subtorus 
		$T' \subset T$ and a root subgroup $V \subset \Aut(X)$ with respect
		to $T$ such that $V \cdot T'$ acts with an open orbit on $X$.
		As before, it follows from Proposition~\ref{prop:characterization_of_aff_spherical_var} that
		$\theta(V) \cdot \theta(T')$ acts with an open orbit on $Y$. This implies that
		$\dim(Y) \leq \dim(V) + \dim(T') = \dim(T)$. On the other hand, since
		$\theta(T)$ acts faithfully on $Y$, we get $\dim(T) \leq \dim(Y)$. In summary, 
		$\dim(Y) = \dim(T)$ and thus $Y$ is $\theta(T)$-toric.
\end{proof}

\section{Connection between the weight monoid and the set of weights of root subgroups}
\addtocontents{toc}{\protect\setcounter{tocdepth}{1}}

Throughout the whole section we fix the following

\subsection*{Notation} We denote by $G$ a connected reductive algebraic group, by $B \subset G$ a Borel subgroup
and by $T \subset B$ a maximal torus. Moreover, we denote by $U \subset B$ the unipotent radical of $B$.
Moreover, we denote $\frak{X}(B)_{\RR} = \frak{X}(B) \otimes_{\ZZ} \RR$ where $\frak{X}(B)$ is the character group of $B$. If $X$ is a $G$-variety, recall that we denote
\[
	D(X) = \Bigset{ \lambda \in \frak{X}(B)}{ 
		\begin{array}{l}
		\textrm{there exists a non-trivial $B$-homogeneous} \\
		\textrm{$\GG_a$-action on $X$ of weight $\lambda$} 
		\end{array}
	}
\]
(see Subsec.~\ref{sec.GroupActionsAndVectorfields} for the definition of a $B$-homogeneous $\GG_a$-action).

\medskip

In this section we give for a quasi-affine $G$-spherical 
variety $X$ a description of
the weight monoid $\Lambda^+(X)$ in terms of $D(X)$, see Theorem~\ref{thm.description_of_weight-monoid} below.

\begin{proposition}
	\label{prop.weights_Vec_ln^U}
	Let $X$ be an irreducible quasi-affine variety with a faithful $G$-action such that $\OO(X)$ is
	a finitely generated $\kk$-algebra. If $G \neq T$, then there is a $\lambda \in D(X)$ with
	\[
		\lambda + \Lambda^+(X) \subset D(X) \quad \textrm{and} \quad
		\Lambda^+(X)_{\infty} = D(X)_{\infty},
	\]
	where the asymptotic cones are taken inside $\frak{X}(B)_{\RR}$.
\end{proposition}

\begin{proof}
	We denote $D = D(X)$.
	By Lemma~\ref{lem:D(X)_contained_in_the_weights_of_Ve^H(X)} we have
	\[
		D \subset 
		\Bigset{\lambda \in \frak{X}(B)}
		{
		  \begin{array}{l}
			\textrm{there is a non-zero vector field in $\Ve^U(X)$} \\
			\textrm{that is normalized by $B$ of weight $\lambda$}	
		  \end{array}
		} =: D' \, .
	\] 
	By Corollary~\ref{cor.VecUX_finitely_generated}
	we know that $\Ve^U(X)$ is finitely generated as an $\OO(X)^U$-module.
	Hence, there are finitely many non-zero $B$-homogeneous 
	$\xi_1, \ldots, \xi_k \in \Ve^U(X)$ such that the $B$-module homomorphism
	\[
		\pi \colon \bigoplus_{i=1}^k \OO(X)^U \xi_i \to \Ve^U(X) \, , \quad
		(r_1 \xi_1, \ldots, r_n \xi_k) \mapsto r_1 \xi_1 + \ldots + r_k \xi_k
	\]
	is surjective. Let $\lambda \in D'$ and let $\eta \in \Ve^U(X)$ be a non-zero vector field
	that is normalized by $B$ of weight $\lambda$. Thus $M = \pi^{-1}(\kk \eta)$ is a rational 
	$B$-submodule of $\bigoplus_{i=1}^k \OO(X)^U \xi_i$ (see Proposition~\ref{prop.locally_finite_and_rational}).
	As each element in $M$ can be written as a sum of $T$-semi-invariants, as $U$ acts trivially on $M$
	and as $\frak{X}(U)$ is trivial, it follows that each element in $M$ can be written as a sum 
	of $B$-semi-invariants. Hence, there is a non-zero $B$-semi-invariant $\xi \in M$ such that 
	$\pi(\xi) = \eta$. As a consequence, the weight of $\xi$ is $\lambda$. Thus we proved that $D'$
	is contained in the weights of non-zero $B$-semi-invariants of $\bigoplus_{i=1}^k \OO(X)^U \xi_i$, i.e.
	\[
		D' \subset \bigcup_{i=1}^k \left( \lambda_i + \Lambda^+(X) \right) \, ,
	\]
	where $\lambda_i \in \frak{X}(B)$ denotes the weight of $\xi_i$.
	
	Since $G \neq T$, we get by Lemma~\ref{lem.B_contains_unipotent_elements} that 
	$U \neq \{e\}$. Since $G$ (and therefore $U$) acts faithfully on $X$, there is a non-trivial 
	$B$-homogeneous	
	$\GG_a$-action $\rho \colon \GG_a \times X \to X$ of a certain weight $\lambda \in D$
	associated to a root subgroup with respect to $B$ in the centre of $U$. 
	Now, we claim that
	\[
		\lambda + \Lambda^+(X) \subset D \, .
	\]
	Indeed, this follows  since for every non-zero $B$-semi-invariant $r \in \OO(X)^U$ of weight 
	$\lambda' \in \frak{X}(B)$,
	the $\GG_a$-action
	\[
		\GG_a \times X \to X \, , \quad (t, x) \mapsto \rho(r(x)t, x)
	\] 
	is non-trivial and $B$-homogeneous of weight $\lambda + \lambda' \in \frak{X}(B)$.
	
	In summary, we have proven
	\[
		\lambda + \Lambda^+(X) \subset D \subset D' \subset \bigcup_{i=1}^k \left( \lambda_i + \Lambda^+(X) \right) \subset \frak{X}(B)_{\RR} \, .
	\]
	From Lemma~\ref{lem.Properties_asymptotic_cone} it follows now that
	$\Lambda^+(X)_{\infty} \subset D_{\infty} \subset D'_{\infty} = \Lambda^+(X)_{\infty}$.
\end{proof}

\begin{theorem}
	\label{thm.description_of_weight-monoid}
	Let $X$ be a quasi-affine $G$-spherical variety which is non-isomorphic to a torus.
	If $G \neq T$ or $X_{\aff} \not\simeq \AA^1 \times (\AA^1 \setminus \{0\})^{\dim(X)-1}$, then
	\[
			\label{eq.closed_formula}
			\tag{$\varocircle$}
			\Lambda^+(X) = \Conv ( D(X)_{\infty}) \cap \Span_{\ZZ}(D(X))
	\]
	where the asymptotic cones and linear spans are taken inside $\frak{X}(B)_{\RR}$.
	Moreover, $\dim \Conv ( D(X)_{\infty}) = \dim \Span_{\RR}(D(X))$ and $D(X)$ is non-empty.
%
\end{theorem}

\begin{proof}
	As in the last proof, we set $D = D(X)$. We get $D \neq \varnothing$, indeed: if $G \neq T$, this follows
	from Lemma~\ref{lem.B_contains_unipotent_elements} and if $G = T$, this follows from Proposition~\ref{prop.DX_generates_M_as_a_group} (as $X$ is not a torus).
	
	Since $X$ is a quasi-affine $G$-spherical variety, it follows from Lemma~\ref{lem:Spherical_quasi-affine}
	that $X_{\aff} = \Spec \OO(X)$ is an affine $G$-spherical variety.
	In particular, $\OO(X)$ is an integrally closed domain, that is finitely generated as a $\kk$-algebra.
	Hence $\OO(X)^U$ is integrally closed and
	it is finitely generated as a $\kk$-algebra (by Proposition~\ref{prop:finite_generation_of_invariant}).
	Since $B$ acts with an open orbit on $X_{\aff}$, the algebraic quotient
	$X_{\aff} \aquot U = \Spec \OO(X)^U$ is an affine $T'$-toric variety where $T'$ is a quotient torus 
	of $T$. Thus we get a natural inclusion of character groups
	\[
		\frak{X}(T') \subset \frak{X}(T) = \frak{X}(B)
	\]
	where we identify $\frak{X}(B)$ with $\frak{X}(T)$ via the restriction homomorphism.
	Using the above inclusion, $\Lambda^+(X)$ is contained  
	inside $\frak{X}(T')$ and it is equal to the set of $T'$-weights of non-zero $T'$-semi-invariants
	of $\OO(X)^U$. As $X_{\aff} \aquot U$ is $T'$-toric, $\Lambda^+(X)$ is a finitely generated semi-group and $\Conv(\Lambda^+(X))$
	is a convex rational polyhedral cone inside $\frak{X}(T')_{\RR} \subset \frak{X}(B)_{\RR}$. 
	Moreover, $\Lambda^+(X)$ generates $\frak{X}(T')$ as a group inside $\frak{X}(B)$ and
	$\Lambda^+(X)$ is saturated in $\frak{X}(T')$, i.e. 
	\[
		\Lambda^+(X) = \Conv(\Lambda^+(X)) \cap \frak{X}(T')
	\]
	(see \cite[Ex. 1.3.4 (a)]{CoLiSc2011Toric-varieties}). Using the inclusion 
	$\frak{X}(T') \subset \frak{X}(T) = \frak{X}(B)$ again, we get $D \subset \frak{X}(T')$, since each
	$B$-homogeneous $\GG_a$-action on $X$ induces a $T'$-homogeneous $\GG_a$-action on
	$X_{\aff} \aquot U$. We distinguish two cases:
	
	\begin{itemize}[leftmargin=*]
	\item $G \neq T$. By Proposition~\ref{prop.weights_Vec_ln^U}, we get inside $\frak{X}(B)_{\RR}$
	\[		
		\Lambda^+(X)_{\infty} = D_{\infty}
	\]
	and there is a $\lambda \in D$ with $\lambda + \Lambda^+(X) \subset D \subset \frak{X}(T')$.
	Since $\Lambda^+(X)$ generates the group $\frak{X}(T')$, we get thus $\Span_{\ZZ}(D) = \frak{X}(T')$.
	As $\Conv(\Lambda^+(X))$ is a rational convex polyhedral cone, we get
	$\Conv(\Lambda^+(X)) = \Conv(\Lambda^+(X)_{\infty})$.
	In summary we have
	\begin{align*}
		\Lambda^+(X) = \Conv(\Lambda^+(X)) \cap \frak{X}(T') 
		&= \Conv(\Lambda^+(X)_\infty) \cap \frak{X}(T') \\
		&= \Conv ( D_{\infty}) \cap \Span_{\ZZ}(D)
	\end{align*}
	and thus~\eqref{eq.closed_formula} holds. The second statement follows now from
	\[
		\dim \Span_{\RR}(D) = \dim T' = \rank \Lambda^+(X) \leq
		\dim \Conv(D_\infty) \leq \dim T' \, .
	\]
	
	\item $G = T$. In particular, $T$ acts faithfully with an open orbit on $X$. Thus
	$T' = T$ and both varieties $X$, $X_{\aff} = X_{\aff} \aquot U$ are $T$-toric.
	
	Denote by $\sigma \subset \Hom_{\ZZ}(\frak{X}(T), \RR)$ the strongly convex
	rational polyhedral cone that describes $X_{\aff}$ and let 
	$\sigma^\vee \subset \frak{X}(T)_{\RR}$ be the dual of $\sigma$.
	By Corollary~\ref{cor.Weights}
	\[
		\label{eq.D_infty}
		\tag{$\triangle$}
		D_{\infty} = \sigma^\vee \setminus \inter(\sigma^\vee)
	\]
	where $\inter(\sigma^\vee)$ denotes the interior of $\sigma^\vee$ inside $\frak{X}(T)_{\RR}$.
	By assumption, $X_{\aff} \not\simeq \AA^1 \times (\AA^1 \setminus \{0\})^{\dim(X)-1}$. This implies that 
	$\dim \sigma > 1$ and we may write $\sigma^\vee = C \times W$ where 
	$C \subset \frak{X}(T)_\RR$
	is a strongly convex polyhedral cone of dimension $> 1$ 
	and $W \subset \frak{X}(T)_\RR$ is a linear subspace. 
	Hence, $C$ is the convex hull of its codimension one faces and thus the same holds 
	for $\sigma^\vee$. Using~\eqref{eq.D_infty}, we get
	\[
		\Conv(D_{\infty}) = \sigma^\vee = \Conv(\Lambda^+(X)) \, .
	\]
	Since $\Lambda^+(X)$ is saturated in $\frak{X}(T)$, the above equality implies that
	\[
		\Lambda^+(X) = \Conv(\Lambda^+(X)) \cap \frak{X}(T) = 
		\Conv(D_\infty) \cap \frak{X}(T) \, .
	\]
	It follows from Proposition~\ref{prop.DX_generates_M_as_a_group} that 
	$\frak{X}(T) = \Span_{\ZZ}(D)$ (here we use that $X \not\simeq T$)
	and thus~\eqref{eq.closed_formula} holds. The second statement follows now from
	\[
		\dim \Span_{\RR}(D) = \dim T = \rank \Lambda^+(X) \leq
		\dim \Conv(D_\infty) \leq \dim T \, .
	\]
	\end{itemize}
\end{proof}

\begin{remark}
	Assume that $G = T$ and that $X$ is a $T$-toric quasi-affine variety.
	Then one could recover the extremal rays of the strongly convex rational polyhedral
	cone that describes $X_{\aff}$ from $D(X)$ in a similar 
	way as in~\cite[Lemma~6.11]{LiReUr2018Characterization-o} by using Corollary~\ref{cor.Weights_rho}.
	In particular, one could then recover $\Lambda^+(X)$ from $D(X)$. 
	However, we wrote Theorem~\ref{thm.description_of_weight-monoid} in order to have a nice 
	``closed formula'' 
	of $\Lambda^+(X)$ in terms of $D(X)$ for almost all quasi-affine $G$-spherical varieties.
\end{remark}


\begin{corollary}
	\label{cor.weights_describe_weight-monoid}
	Let $X$ be a quasi-affine $G$-spherical variety. Then the following holds
	(here, the linear spans and the asymptotic cones
	are taken inside $\frak{X}(B)_{\RR}$):
	\begin{enumerate}[leftmargin=*]
		\item \label{cor.weights_describe_weight-monoid1}
				If $\dim \Conv(D(X)_{\infty}) = \dim \Span_{\RR}(D(X))$ and $D(X)$ is non-empty, then
				\[
					\Lambda^+(X) = \Conv ( D(X)_{\infty}) \cap \Span_{\ZZ}(D(X)) \, .
				\]
		\item \label{cor.weights_describe_weight-monoid2}
				If $\dim \Conv(D(X)_{\infty}) < \dim \Span_{\RR}(D(X))$ and if $D(X)$ is non-empty, then
				$D(X)_{\infty}$ is a hyperplane in $\Span_{\RR}(D(X))$ and
				\[
					\Lambda^+(X) = H^+ \cap \Span_{\ZZ}(D(X))
				\]
				where $H^+ \subset \Span_{\RR}(D(X))$
				is the closed half space with boundary $D(X)_{\infty}$ 
				that does not intersect $D(X)$.
		\item \label{cor.weights_describe_weight-monoid3}
				If $D(X)$ is empty, then $\Lambda^+(X) = \frak{X}(T)$.
	\end{enumerate}
	In particular, the following holds: 
	If $Y$ is another quasi-affine $G$-spherical variety with $D(Y) = D(X)$, then 
	$\Lambda^+(Y) = \Lambda^+(X)$.
\end{corollary}

\begin{proof}
	If $X$ is a torus, then $D(X)$ is empty. In particular, $G = T$ by Lemma~\ref{lem.B_contains_unipotent_elements} and thus $X \simeq T$. Hence, $\Lambda^+(X) = \frak{X}(T)$ and we are 
	in case~\eqref{cor.weights_describe_weight-monoid3}. Thus we may assume that $X$ is not a torus.
	
	If $G \neq T$ or $X_{\aff} \not\simeq \AA^1 \times (\AA^1 \setminus \{0\})^{\dim(X)-1}$, then 
	the statement is a direct consequence of Theorem~\ref{thm.description_of_weight-monoid} and we 
	are in case~\eqref{cor.weights_describe_weight-monoid1}.
	
	Thus we may assume that $G = T$ and 
	$X_{\aff} \simeq \AA^1 \times (\AA^1 \setminus \{0\})^{\dim(X)-1}$. In particular, $D(X)$ is non-empty
	and by Proposition~\ref{prop.DX_generates_M_as_a_group} we get $\frak{X}(T) = \Span_{\ZZ}(D(X))$.
	Denote by
	$\sigma \subset \Hom_{\ZZ}(\frak{X}(T), \RR)$ the closed strongly 
	convex rational polyhedral cone that describes $X_{\aff}$.
	In this case $\sigma$ is a single ray and thus
	$\sigma^\vee$ is a closed half space in $\frak{X}(T)_\RR$.
	As $D(X)_\infty = \sigma^\vee \setminus \inter(\sigma^\vee)$ (see Corollary~\ref{cor.Weights}),
	it follows that $D(X)_\infty$ is a hyperplane in $\Span_{\RR}(D(X))$.
	By definition $\Lambda^+(X) = \sigma^\vee \cap \Span_{\ZZ}(D(X))$ and
	$\sigma^\vee$ is in fact the closed half space with boundary
	$D(X)_\infty$ that does not intersect $D(X)$ (see Corollary~\ref{cor.Weights_rho}).
	In particular, $\dim \Conv(D(X)_\infty) < \dim T = \dim \Span_{\RR}(D(X))$
	and thus we are in case~\eqref{cor.weights_describe_weight-monoid2}.
\end{proof}

As a consequence of Corollary~\ref{cor.weights_describe_weight-monoid} we prove that
for a $G$-spherical variety $X$ the weight monoid $\Lambda^+(X) \subseteq \frak{X}(B)$ 
is determined by its automorphism group.

\begin{corollary}\label{thesameweights}
	\label{cor.Autos_describe_weight_monoids}
	Let $X, Y$ be irreducible quasi-affine normal varieties. Assume that 
	$X$ is $G$-spherical, $X$ is different from an algebraic torus and that
	there exists an isomorphism of groups 
	$\theta \colon \Aut(X) \simeq \Aut(Y)$ that preserves 
	algebraic subgroups.
	Then $Y$ is $G$-spherical for the $G$-action induced by $\theta$ 
	and $\Lambda^+(X) = \Lambda^+(Y)$.
\end{corollary}

\begin{proof}
	The first claim follows from  Proposition~\ref{prop.iso_and_sphericity}. 
	To show that $\Lambda^+(X) = \Lambda^+(Y)$ let us denote by
	 $D(X), D(Y) \subset \frak{X}(B)$  the set of $B$-weights of non-trivial $B$-homo\-geneous $\GG_a$-actions
	on $X$ and $Y$ respectively. We get $D(X) = D(Y)$ from Lemma~\ref{lem.auto_perserves_weights}.
	Now, Corollary~\ref{cor.weights_describe_weight-monoid} implies $\Lambda^+(X) = \Lambda^+(Y)$.
\end{proof}

\begin{theorem}
	\label{thm.smooth_spherical_var}
	Let $X$ and $Y$ be irreducible normal affine varieties. Assume that 
	$X$ is $G$-spherical and that $X$ is not isomorphic 
	to a torus. Moreover, we assume that there is an isomorphism  of groups 
	$\theta \colon \Aut(X) \simeq \Aut(Y)$ that preserves 
	algebraic subgroups. We consider $Y$ as a $G$-variety by the induced action via $\theta$. Then
	$X$, $Y$ are isomorphic as $G$-varieties, provided one of the following statements holds
	\begin{itemize}
		\item $X$ and $Y$ are smooth or
		\item $G = T$ is a torus.
 	\end{itemize}
\end{theorem}

\begin{proof}
	By Corollary \ref{cor.Autos_describe_weight_monoids}, $Y$ is $G$-spherical
	and  the weight monoids
	$\Lambda^+(X)$ and $\Lambda^+(Y)$  coincide.
	In case $X$ and $Y$ are smooth, the statement follows now from the affirmative
	answer of \name{Knop}'s Conjecture, see \cite[Theorem~1.3]{Lo2009Proof-of-the-Knop-}.
	In case $G$ is a torus, it is classical, that from the weight monoid $\Lambda^+(X)$
	one can reconstruct the toric variety $X$ up to $G$-equivariant isomorphisms,
	see e.g. Fulton \cite[\S1.3]{Fu1993Introduction-to-to}.
\end{proof}

\section{A counter example}
\label{sec.Counter_example}

For th rest of this article, we give an example which shows that we cannot drop the normality condition
in Main Theorem~\ref{mainthm:autos}. The example is borrowed from \cite{Re2017Characterization-o}.

Let $\mu_d \subset \kk^\ast$ be the 
finite cyclic subgroup of order $d$ and let it act on $\mathbb{A}^n$ via 
$t \cdot (x_1, \ldots, x_n) = (t x_1, \ldots t x_n)$. The algebraic quotient 
$\mathbb{A}^n/\mu_d$ has the coordinate ring 
\[
	\mathcal{O}(\mathbb{A}^n/\mu_d) = \bigoplus_{k \ge 0} \kk[x_1,\dots,x_n]_{kd}
	\subset \kk[x_1, \ldots, x_n] \, , 
\]
where 
$\kk[x_1,\dots,x_n]_{i} \subset \kk[x_1, \ldots, x_n]$ 
denotes the subspace of homogeneous polynomials of degree $i$. For each $s \geq 2$, let
\[
	A_{d, n}^s = \Spec\left(\kk \oplus  \bigoplus_{k \ge s} \kk[x_1,\dots,x_n]_{kd} \right) \, .
\]

\begin{lemma}
	\label{lem.ex_normality}
	For $n, d, s \geq 2$ and the algebraic quotient $\pi \colon \AA^n \to \AA^n / \mu_d$ holds:
	\begin{enumerate}
		\item \label{lem.ex_normality1} 
		The variety
		$\mathbb{A}^n/\mu_d$ is $\SL_n(\kk)$-spherical for the natural $\SL_n(\kk)$-action
		on $\AA^n$ and $\AA^n / \mu_d$ is smooth outside $\pi(0, \ldots, 0)$;
		\item \label{lem.ex_normality2} 
		There is an $\SL_n(\kk)$-action on  $A_{d, n}^s$ such that the morphism $\eta \colon \mathbb{A}^n/\mu_d \to A_{d, n}^s$ which is induced by the natural
		inclusion $\OO(A_{d. n}^s) \subset \OO(\mathbb{A}^n/\mu_d)$ is 
		$\SL_n(\kk)$-equivariant. Moreover, $\eta$ is the normalization map and it is bijective;
		\item \label{lem.ex_normality3} 
		The natural group homomorphism $\Aut(A_{d, n}^s) \to \Aut(\mathbb{A}^n/\mu_d)$
		is a group isomorphism that preserves algebraic subgroups;		
		\item \label{lem.ex_normality4}
		The variety $A_{d, n}^s$ is not normal;
		\item \label{lem.ex_normality5}
		The weight monoids $\Lambda^+(A_{d, n}^s)$ and $\Lambda^+(\mathbb{A}^n/\mu_d)$ inside $\frak{X}(B)$ are distinct when we fix
		a Borel subgroup $B \subset \SL_n(\kk)$.
	\end{enumerate}
\end{lemma}

\begin{proof}
	\eqref{lem.ex_normality1}: As the natural $\SL_n(\kk)$-action on $\AA^n$ commutes with the $\mu_d$-action,
	we get an induced $\SL_n(\kk)$-action on $\AA^n / \mu_d$ such that $\pi$ is $\SL_n(\kk)$-equivariant
	and $\AA^n / \mu_d$ is $\SL_n(\kk)$-spherical. As $\SL_n(\kk)$ acts transitively
	on $\AA^n \setminus \{ 0 \}$, the projection $\pi$ induces a finite \'etale morphism 
	$\AA^n \setminus \{(0, \ldots, 0) \} \to (\AA^n / \mu_d) \setminus \{ \pi(0, \ldots, 0 )\}$. This shows
	that $(\AA^n / \mu_d) \setminus \{ \pi(0, \ldots, 0 )\}$ is smooth. 
%

	\eqref{lem.ex_normality2}: 
	As $\SL_n(\kk)$ acts linearly on $\AA^n$, we get an $\SL_n(\kk)$-action on 
	$A_{d, n}^s$ such that $\eta \colon \mathbb{A}^n/\mu_d \to A_{d, n}^s$ is $\SL_n(\kk)$-equivariant.
	
	As $\AA^n$ is normal, the algebraic quotient $\AA^n / \mu_d$ is normal.
	As $\OO(A_{d, n}^s)$ has finite codimension in $\OO(\AA^n / \mu_d)$, the ring
	extension $\OO(A_{d, n}^s) \subset \OO(\AA^n / \mu_d)$
	is an integral. Moreover, for each 
	monomial $f \in \kk[x_1, \ldots, x_n]$ of degree $sd$, we get an
	equality by localizing, namely $\OO(A_{d, n}^s)_f = \OO(\AA^n / \mu_d)_f$, and thus $\eta$ is birational.
	This shows that $\eta$ is the normalization map.
	Moreover, as $\eta$ is $\SL_n(\kk)$-equivariant
	and as $\SL_n(\kk)$ acts transitively on $(\AA^d / \mu_d) \setminus \{ \pi(0, \ldots, 0) \}$, we get
	that $A_{d, n}^s \setminus \{ \eta(\pi(0, \ldots, 0))\}$ is smooth and as $\eta$ is the normalization,
	it is an isomorphism over the complement of $\eta(\pi(0, \ldots, 0))$.
	Moreover, $\eta^{-1}(\eta(\pi(0, \ldots, 0))) = \{ \pi(0, \ldots, 0) \}$ and thus $\eta$ is bijective.
	
	\eqref{lem.ex_normality3}:
	Each automorphism of $A_{d, n}^s$ lifts uniquely to an automorphism of $\AA^n / \mu_d$ via
	the normalization $\AA^n / \mu_d \to A_{d, n}^s$ and hence we get an
	injective group homomorphism $\theta \colon \Aut(A_{d,n}^s) \to \Aut(\mathbb{A}^n/\mu_d)$,
	see \cite{Ra1964A-note-on-automorp}.
	
	Now we prove that $\theta$ is surjective. For this, let $\varphi \in \Aut(\AA^n / \mu_d)$. 
	As $n \geq 2$, 
	the algebraic quotient $\AA^n \to \AA^n/\mu_d$ is in fact the Cox realization of the toric variety
	$\AA^n/\mu_d$ (see~\cite[Theorem 3.1]{ArGa2010Cox-rings-semigrou}). 
	By~\cite[Corollary 2.5, Lemma~4.2]{Be2003Lifting-of-morphis},
	$\varphi$ lifts via $\AA^n \to \AA^n / \mu_d$ to an automorphism $\psi$ of $\AA^n$ 
	and there is an integer  $c \geq 1$ which is coprime to $d$ such that for each $t \in \mu_d$
	and each $(a_1, \ldots, a_n) \in \AA^n$ we have
	\[
		\psi(t a_1, \ldots, t a_n) = t^c \psi(a_1, \ldots, a_n) \, .
	\]
	This implies that for each $i \in \{ 1, \ldots, n \}$,
	\[
		\psi^\ast(x_i) \in \bigoplus_{k\geq0} \kk[x_1,\dots,x_n]_{kd+c} \, .
	\]
	As $\psi$ is an automorphism of $\AA^n$, we get $c = 1$ and thus
	$\psi$ is $\mu_d$-equivariant (see also \cite[Proposition~4]{Re2017Characterization-o}).
	Hence, $\psi^\ast \colon \kk[x_1, \ldots,x_n] \to \kk[x_1, \ldots, x_n]$ maps
	$\OO(A_{d, n}^s)$ onto itself and by construction restricts to $\varphi^\ast$ on $\OO(\AA^n / \mu_d)$.
	Therefore, there is an endomorphism $\tilde{\varphi} \colon A_{d, n}^s \to A_{d, n}^s$
	that induces $\varphi \in \Aut(\AA^n / \mu_d)$
	via the normalization map $\eta \colon \AA^n / \mu_d \to A_{d, n}^s$. 
	As $\eta$ and $\varphi$ are bijective, 
	$\tilde{\varphi}$ is bijective as well;
	hence $\tilde{\varphi}$ is an automorphism of $A_{d, n}^s$ by \cite[Lemma~1]{Ka2005On-a-theorem-of-Ax} and thus $\theta$ is surjective.
	
	Since $\theta \colon \Aut(A_{d, n}^s) \to  \Aut(\AA^n / \mu_d)$ is a group isomorphism 
	and as it is induced by the normalization map $\AA^n / \mu_d \to A_{d, n}^s$, it follows that $\theta$
	is an isomorphism of ind-groups, see~\cite[Proposition 12.1.1]{FuKr2018On-the-geometry-of}. In particular, $\theta$ is a group isomorphism that
	preserves algebraic subgroups.
	
	\eqref{lem.ex_normality4}: The normalization map $\AA^n / \mu_d \to A_{d, n}^s$ is
	not an isomorphism, since the inclusion $\OO(A_{d, n}^s) \subset \OO(\AA^n / \mu_d)$ is proper
	(note that $s \geq 2$).
	
	\eqref{lem.ex_normality5}:
	We may assume hat $B \subset \SL_n(\kk)$ is the Borel subgroup of upper triangular matrices.
	Denote by $U \subset B$ the unipotent radical, i.e. the upper triangular matrices with $1$ on the diagonal.
	Then the subrings of $U$-invariant functions satisfy
	$\OO(\AA^n / \mu_d)^U = \bigoplus_{k \geq 0} \kk x_n^{kd}$ and 
	$\OO(A_{d, n}^s)^U = \kk \oplus \bigoplus_{k \geq s} \kk x_n^{kd}$. Denote by
	$\chi_n \colon B  \to \GG_m$ the character which is the projection to the entry $(n, n)$.
	Then we get
	\[
		\Lambda^+(\AA^n / \mu_d) = \set{\chi_n^{kd}}{k \geq 0} \quad \textrm{and} \quad
		\Lambda^+(A_{d, n}^s) = \set{\chi_n^{kd}}{k = 0 \ \textrm{or} \ k \geq s}
	\]
	inside $\frak{X}(B)$ and as $s \geq 2$, these monoids are distinct.
\end{proof}

\par\bigskip
\renewcommand{\MR}[1]{}
\bibliographystyle{amsalpha}

\providecommand{\bysame}{\leavevmode\hbox to3em{\hrulefill}\thinspace}
\providecommand{\MR}{\relax\ifhmode\unskip\space\fi MR }
\providecommand{\MRhref}[2]{%
  \href{http://www.ams.org/mathscinet-getitem?mr=#1}{#2}
}
\providecommand{\href}[2]{#2}

\end{document}